\documentclass[11pt,leqno]{amsart}
\usepackage{amsthm,amsfonts,amssymb,amsmath,oldgerm}
\numberwithin{equation}{section}
\usepackage{fullpage}
\usepackage{amsmath}
\usepackage{amsfonts}
\usepackage{amssymb}
\usepackage{setspace}
\usepackage{stmaryrd}
\usepackage{mathrsfs}
\usepackage{fancyhdr}
\usepackage{graphicx}

\usepackage{psfrag}
\usepackage{listings} 

\usepackage[top=30mm,bottom=30mm,left=25mm,right=25mm,a4paper]{geometry}

\usepackage{hyperref}

\renewcommand\d{\partial}
\newcommand\dD{\textrm{d}}
\newcommand\eD{\textrm{e}}
\newcommand\iD{\textrm{i}}
\newcommand\ii{\iD}
\newcommand\dd{\dD}
\newcommand\bdD{\textrm{{\bf d}}}
\def\eps{\varepsilon }

\newcommand{\Id}{{\rm Id}}
\newcommand{\I}{{\rm I}}

\newcommand\br{\begin{remark}}
\newcommand\er{\end{remark}}
\newcommand\bp{\begin{pmatrix}}
\newcommand\ep{\end{pmatrix}}
\newcommand{\be}{\begin{equation}}
\newcommand{\ee}{\end{equation}}
\newcommand\ba{\begin{equation}\begin{aligned}}
\newcommand\ea{\end{aligned}\end{equation}}
\newcommand\ds{\displaystyle}

\newcommand{\beg}{\begin{example}}
\newcommand{\eeg}{\end{exaplem}}
\newcommand{\bpr}{\begin{proposition}}
\newcommand{\epr}{\end{proposition}}
\newcommand{\bt}{\begin{theorem}}
\newcommand{\et}{\end{theorem}}
\newcommand{\bc}{\begin{corollary}}
\newcommand{\ec}{\end{corollary}}
\newcommand{\bl}{\begin{lemma}}
\newcommand{\el}{\end{lemma}}
\newcommand{\bd}{\begin{definition}}
\newcommand{\ed}{\end{definition}}
\newcommand{\brs}{\begin{remarks}}
\newcommand{\ers}{\end{remarks}}

\newtheorem{theorem}{Theorem}[section]
\newtheorem{proposition}[theorem]{Proposition}
\newtheorem{corollary}[theorem]{Corollary}
\newtheorem{lemma}[theorem]{Lemma}

\theoremstyle{remark}
\newtheorem{remark}[theorem]{Remark}
\theoremstyle{definition}
\newtheorem{definition}[theorem]{Definition}

\newtheorem{example}[theorem]{Example}

\newcommand\R{\mathbf R}
\newcommand\C{\mathbf C}
\newcommand{\N}{\mathbf N}
\newcommand{\Z}{\mathbf Z}

\newcommand{\cn}{\operatorname{cn}}
\newcommand{\one}{\mathbf{1}}
\newcommand{\grandO}{\cO}

\newcommand{\lambdae}[1]{\lambda_{{\rm e},#1}}
\newcommand{\lambdap}[1]{\lambda_{{\rm p},#1}}
\newcommand{\Lambdae}{\lambda_{{\rm e}}}
\newcommand{\Lambdap}{\lambda_{{\rm p}}}
\newcommand{\etae}[1]{\eta_{{\rm e},#1}}

\newcommand{\phie}[1]{\phi_{{\rm e},#1}}
\newcommand{\phip}[1]{\phi_{{\rm p},#1}}
\newcommand{\tphie}[1]{\tphi_{{\rm e},#1}}
\newcommand{\tphip}[1]{\tphi_{{\rm p},#1}}
\newcommand{\Phie}{\Phi_{{\rm e}}}
\newcommand{\Phip}{\Phi_{{\rm p}}}
\newcommand{\tPhie}{\tPhi_{{\rm e}}}
\newcommand{\tPhip}{\tPhi_{{\rm p}}}
\newcommand{\Ke}{K_{{\rm e}}}
\newcommand{\Kp}{K_{{\rm p}}}
\newcommand{\Ko}[1]{K_{{\rm o},#1}}
\newcommand{\tKo}[1]{\tK_{{\rm o},#1}}
\newcommand{\ko}[1]{k_{{\rm o},#1}}
\newcommand{\Ge}{G_{{\rm e}}}
\newcommand{\Gp}{G_{{\rm p}}}
\newcommand{\Go}[1]{G_{{\rm o},#1}}
\newcommand{\go}[1]{g_{{\rm o},#1}}
\newcommand{\tGo}[1]{\tG_{{\rm o},#1}}
\newcommand{\Ae}[1]{A_{{\rm e},#1}}
\newcommand{\Ap}[1]{A_{{\rm p},#1}}
\newcommand{\Ao}[1]{A_{{\rm o},#1}}

\newcommand{\Up}[1]{\cU^{#1}}

\newcommand{\tSe}{\tS_{{\rm e}}}
\newcommand{\tSo}{\tS_{{\rm o}}}

\newcommand{\spp}{s^{\rm p}}

\newcommand{\SigW}{\Sigma^{W}}

\newcommand\bC{{\mathbf C}}

\newcommand\bU{{\mathbf U}}

\newcommand\ubU{{\underline \bU}}

\newcommand\uM{{\underline M}}

\newcommand\uP{{\underline P}}

\newcommand\uU{{\underline U}}

\newcommand\ua{{\underline a}}

\newcommand\uc{{\underline c}}

\newcommand\uk{{\underline k}}

\newcommand\upsi{{\underline \psi}}
\newcommand\uom{{\underline \omega}}

\newcommand\cB{{\mathcal B}}
\newcommand\cC{{\mathcal C}}

\newcommand\cM{{\mathcal M}}
\newcommand\cN{{\mathcal N}}
\newcommand\cO{{\mathcal O}}
\newcommand\cP{{\mathcal P}}

\newcommand\cU{{\mathcal U}}

\newcommand\tq{{\widetilde q}}

\newcommand\tbeta{{\widetilde \beta}}

\newcommand\tphi{{\widetilde \phi}}

\newcommand\tchi{{\widetilde \chi}}
\newcommand\tpsi{{\widetilde \psi}}

\newcommand\tG{\widetilde{G}}

\newcommand\tK{\widetilde{K}}
\newcommand\tL{\widetilde{L}}
\newcommand\tM{\widetilde{M}}

\newcommand\tP{\widetilde{P}}

\newcommand\tR{\widetilde{R}}
\newcommand\tS{\widetilde{S}}

\newcommand\tV{\widetilde{V}}
\newcommand\tW{\widetilde{W}}

\newcommand\tPhi{\widetilde{\Phi}}

\title{
Linear asymptotic stability and modulation behavior near periodic waves of the Korteweg--de Vries equation
}

\author{L.~Miguel Rodrigues}
\address{
Universit\'e de Rennes 1,
IRMAR, UMR CNRS 6625,
263 avenue du General Leclerc;
F-35042 Rennes Cedex, FRANCE}
\email{{\tt luis-miguel.rodrigues@univ-rennes1.fr}}
\thanks{Research of L.Miguel Rodrigues was partially supported by the ANR project
BoND ANR-13-BS01-0009-01.\\
}

\begin{document}

\begin{abstract}
We provide a detailed study of the dynamics obtained by linearizing the Korteweg--de Vries equation about one of its periodic traveling waves, a cnoidal wave. In a suitable sense, linearly analogous to space-modulated stability \cite{JNRZ-conservation}, we prove global-in-time bounded stability in any Sobolev space, and asymptotic stability of dispersive type. Furthermore, we provide both a leading-order description of the dynamics in terms of slow modulation of local parameters and asymptotic modulation systems and effective initial data for the evolution of those parameters. This requires a global-in-time study of the dynamics generated by a non normal operator with non constant coefficients. On the road we also prove estimates on oscillatory integrals particularly suitable to derive large-time asymptotic systems that could be of some general interest.

\vspace{1em}

%\begin{center}

\vspace{1em}

%\begin{center}
\end{abstract}

%\date{\today}
\maketitle

%%%%%%%%%%%%%%%%%%%%%%%%%%%%%%%%%%%%%%%%%%%%%%%%%%%%%% 
%\end{center}

%\clearpage
%\tableofcontents
%\clearpage

%%%%%%%%%%%%%%%%%%%%%%%%%%%%%%%%%%%%%%%%%%%%%%%%%%%%%%         
%       INTRODUCTION                                 %
%%%%%%%%%%%%%%%%%%%%%%%%%%%%%%%%%%%%%%%%%%%%%%%%%%%%%%

\section{Introduction}\label{s:introduction}

Substantial efforts have been recently devoted to complete a rather comprehensive theory analyzing the dynamics near periodic waves of dissipative systems of partial differential equations. We refer the reader to \cite{JNRZ-conservation,R,R_Roscoff} for a thorough account of the available parabolic analysis. Comparatively, the parallel analysis of dispersive Hamiltonian systems seems still in its infancy. The goal of the present contribution to a general theory, still to come, is to show, on a case-study, how for the \emph{linearized} dynamics one may completely recover the fine description of parabolic cases.

\subsection{Preliminary observations}

As a preliminary warning we strongly emphasize that we always consider waves as solutions of \emph{extended} systems. This is perfectly common when dealing with asymptotically constant wave profiles --- corresponding to solitary waves, kinks, shocks... --- but is still rather unusual in the dynamical\footnote{As opposed to the literature devoted to spectral studies.} literature when focusing on periodic waves. We believe however that this is the right way to capture some features of wave propagation and in particular to incorporate the rich multi-scale space-time dynamics expected to occur near periodic waves. References on an alternative point of view focusing on bounded domains with \emph{periodic boundary conditions} --- and which has a long and successful history --- may be found in \cite{Angulo-Pava,KapitulaPromislow_book,Benzoni-Mietka-Rodrigues}. Another related feature of our point of view is that we do not break invariance by translation by either focusing on a specific region of the space-time diagram or introducing weighted norms. Though the author does not know any implementation of these strategies in a periodic context\footnote{Excepting those relying on inverse scattering methods.}, for solitary waves and especially those of scalar equations the latter is now a well-established way to effectively bring Hamiltonian equations in a form as parabolic as possible, see \cite{Pego-Weinstein-asymptotic-stability,KapitulaPromislow_book}.

Here we restrict our attention to the Korteweg--de Vries equation as an archetype of dispersive Hamiltonian equation. It also comes with the key advantage of integrability. Indeed, whereas our ultimate goal is to show how to derive dynamical behavior from spectral information and thus to prove theorems where spectral properties are assumed, we believe that it is sounder to do so only after a sufficient number of spectral studies have gathered clear evidence of what is the best notion of spectral stability one may expect. Obviously the latter depends strongly on the nature of the background solution and on the class of system under consideration. We refer the reader to \cite{R,R_Roscoff} for detailed discussions of this question for periodic waves and especially of the notion of \emph{diffusive spectral stability} that has slowly emerged as the relevant set of spectral conditions for periodic waves of dissipative systems. Yet, gathering either analytically or sometimes even numerically the relevant pieces of spectral information may often appear as a daunting task since in general profiles are not known explicitly and almost no a priori knowledge of even a part of the spectrum of relevant operators --- that have variable coefficients that are not asymptotically constant -- is available\footnote{Whereas in the asymptotically constant case the essential spectrum is derived from simple Fourier computations.}. Moreover for the finer dynamical results one may need very precise description of the (critical part of the) spectrum, not usually included in classical spectral studies\footnote{In particular in the Hamiltonian case, since there is no dynamical theory to motivate it...}. The reader may find in \cite{JNRZ-KdV-KS,BJNRZ-KdV-SV} examples of spectral analyses required to apply the abstract parabolic theory. This is where, here, we use integrability to offer relatively elementary proofs of almost all relevant spectral claims and for the remaining ones, gathered in Assumption~\ref{A_intro-bis} below, a way to observe them by well-conditioned soft numerics. On this spectral side we strongly rely on the approach and results of \cite{Bottman-Deconinck}.

Again, we are mostly interested in developing tools to study periodic waves of general dispersive equations so that the use of integrability is here restricted to obtaining relevant spectral information. However for the Korteweg--de Vries equation it is possible to use integrability by inverse scattering in a deeper way and derive large-time asymptotics directly at the nonlinear level, as proved in \cite{Mikikits-PhD,Mikikits-Teschl}. Moreover the decay proved here at the linearized level turns out to be too slow to be used directly in any simple argument yielding a different proof of those nonlinear asymptotics for the Korteweg--de Vries equation. We stress however that the strategy of our proofs seems robust enough to be adapted to cases where one does expect to derive asymptotics for the nonlinear dynamics from bounds on the linearized evolution combined with a priori estimates. To the opinion of the author the analysis of a case where this occurs seems to be the next natural step in laying foundations for a general theory.

\subsection{Bounded stability}

After these preliminary warnings, we now start a precise account of results obtained in the present contribution. We choose the following form
of the Korteweg--de Vries equation (KdV)
\be\label{KdV}
U_t\ +\ \left(\tfrac12\,U^2\right)_x\ +\ U_{xxx}\ =\ 0
\ee
where $U(t,x)$ is scalar and $t$ and $x$ denote time and space variable. A solution to \eqref{KdV} is called a periodic (uniformly traveling) wave if it has the form 
$U(t,x)=\uU(k\,x+\omega\,t)$ for some profile $\uU$ periodic of period one, some wavenumber $k$ and some time frequency $\omega$. The corresponding phase velocity is then $c=-\omega/k$. We shall focus on the dynamics near a given wave with some fixed wavenumber $\uk$ and frequency $\uom$ and hence write \eqref{KdV} in the corresponding moving frame. Introducing $W$ through $U(t,x)\ =\ W(t,\uk\,x+\uom\,t)$ equation \eqref{KdV} becomes
\be\label{KdV-move}
W_t\ +\uom\,W_x\ +\ \uk\,\left(\tfrac12\,W^2\right)_x\ +\ \uk^3\,W_{xxx}\ =\ 0
\ee
so that $\uU$ is a steady solution of \eqref{KdV-move}. Setting $W=\uU+\tW$ and performing a naive linearization yield $\tW_t-L\tW=0$ where
\be\label{linop}
L\,\tW\ =\ -\uom\,\tW_x\ -\ \uk\,\left(\uU\,\tW\right)_x\ -\ \uk^3\,\tW_{xxx}\,.
\ee

We stress however that even if the linear decay obtained below were faster we would not expect $\tW$ to remain small at least in norms encoding some localization so that it is a priori unclear what is the role of this linearization in the description of the large-time dynamics. To be more specific let us recall that at the nonlinear level analyses of parabolic cases suggest that one should not expect to control $\|W-\uU\|_X$, for some reasonable functional space $X$ but instead to bound
$$
\inf_{\Psi\ \textrm{one-to-one}}\|W\circ\Psi-\uU\|_X\,+\,\|\d_x(\Psi-\Id)\|_X\,.
$$
This encodes preservation of shapes but allows for a synchronization of phases by a sufficiently slow phase shift. That the corresponding notion of stability is indeed the relevant notion for periodic waves has slowly emerged along years. In \cite{JNRZ-conservation} this notion --- coined there as \emph{space-modulated stability} --- has been proved to be sharp for general parabolic systems. Moreover in \cite{JNRZ-conservation} have also been identified what are necessary and sufficient cancellations in the structure of parabolic systems --- called there \emph{phase uncoupling} --- to ensure usual orbital stability\footnote{Here this amounts to controlling
$$
\inf_{\Psi\ \textrm{uniform translation}}\|W\circ\Psi-\uU\|_X\,.
$$} instead of space-modulated stability. Though the parabolic machinery obviously does not apply to \eqref{KdV} it is worth mentioning that \eqref{KdV} does not exhibit such null structures.

To unravel what is left of this at the linearized level, let us mimic the first steps of the natural strategy to prove space-modulated stability. Namely, one pick an initial datum $W_0$ and a couple $(V_0,\psi_0)$ such that
$$
W_0\circ(\Id-\psi_0)\ =\ \uU\ +\ V_0
$$
with $(V_0,\d_x\psi_0)$ sufficiently small and try to prove that there exists a corresponding solution $W$ such that there exists $(V,\psi)$ with $(V,\d_x\psi)$ small and $W\circ(\Id-\psi)\ =\ \uU\ +\ V$. The first key observation is that in terms of the sought $(V,\psi)$ equation \eqref{KdV-move} takes the form
$$
(V+\uU_x\psi)_t\ -\ L\,(V+\uU_x\psi)\ =\ \cN[V,\psi_t,\psi_x]\,.
$$
where $L$ is the operator defined above, and $\cN$ is nonlinear in $(V,\psi_t,\psi_x)$ and their derivatives, and, locally, at least quadratic. In particular neglecting terms expected to be at least quadratically small leaves
$$
(V+\uU_x\psi)_t\ -\ L\,(V+\uU_x\psi)\ =\ 0\,.
$$

As readily observed this more involved point of view also leads to the consideration of the group $(S(t))_{t\in\R}$ generated by $L$. But it also shows that \emph{linear space-modulated stability} should be defined by requiring a control of 
\be\label{sm_norm}
N_X(W)\ =\ \inf_{\substack{W=V+\uU_x\psi}}%\\\psi(\infty)=-\psi(-\infty)}}
%\\\psi \textrm{ is low frequency}}} 
\|V\|_{X}\ +\ \|\psi_x\|_{X}
\ee
and not of $\|W\|_X$. Therefore the following result should be interpreted as bounded linear stability of periodic waves of \eqref{KdV} in a space-modulated sense.

\begin{theorem}{\emph{Bounded linear stability.}}\label{th:bounded}
For any $s\in\N$, there exists $C$ such that for any $W_0$ such that $N_{H^s(\R)}(W_0)<\infty$ and any time $t\in\R$
$$
N_{H^s(\R)}(S(t)\,W_0)\ \leq\ C\, N_{H^s(\R)}(W_0)\,.
$$
\end{theorem}

We emphasize that the result is non trivial even for $s=0$ since $L$ is not a normal operator\footnote{We shall observe for instance that zero is a Floquet eigenvalue of $L$ associated with a non trivial Jordan chain.} and $(S(t))_{t\in\R}$ is not expected to be a group of unitary transformations on any subspace of $L^2(\R)$. Since coefficients of $L$ are not constant one should also notice that the proof requires an $s$-by-$s$ analysis\footnote{In contrast with proofs for constant coefficients operators.}. Indeed, in the reverse direction, there is now a large literature devoted to the analysis of \emph{growth} rates of higher-order Sobolev norms for some classes of evolutions that are unitary on $L^2$ --- for instance those generated by linear Schr\"odinger equations including a (time-dependent) potential. Observe also that whereas the Korteweg-de Vries evolution supports an infinite number of conservations laws providing in other contexts a uniform control of Sobolev norms of solutions the kind of initial data considered here does not seem compatible with any form of integration of those conservation laws that would provide conservations of useful functionals so that they do not play any role in our analysis.

\subsection{Asymptotic stability}

We now turn to asymptotic linear stability. Since our focus is mostly on methodology we first recall some facts that are well-known to experts. Note first that the system at hand possesses classical real and Hamiltonian symmetries so that one already knows without any further knowledge that (marginally) spectrally stable\footnote{In the weak sense that the spectrum of the generator of the linearized evolution does not intersect the open right-half plane.} waves come with $L^2$-spectrum that lies on the imaginary axis. For the operator $L$ under study this may be extended to all Sobolev spaces by a rather general functional-analytic argument\footnote{The main nontrivial step is to notice that for operators with compact resolvents, as the Bloch symbols $L_\xi$ introduced in Section~\ref{s:Bloch}, this independence does hold. See \cite[Theorem~4.2.15]{Davies}.} showing that the $W^{s,p}$-spectrum does not depend on $(s,p)$. In particular for none of the waves under consideration one expects exponential decay of $(S(t))_{t\in\R}$, or of any part of it, in $\cB(W^{s,p}(\R))$-norm. But once exponential decay has been ruled out corollaries of the Datko-Pazy theorem \cite[Theorem~3.1.5 \& Corollary 3.1.6]{vanNeerven} preclude any reasonable form of decay in $\cB(W^{s,p}(\R))$-norm for $(S(t))_{t\in\R}$.

In short, in the former paragraph we have recalled why for situations similar to those dealt with here, to exhibit decay one needs to leave the semi-group framework where initial data and solutions at later times are estimated in same norms. Specifically, here, the mechanism expected to yield some form of return to equilibrium is dispersive decay, in particular the trade-off is localization of the data against uniform decay of the solution. Therefore one aims at proving that the solution may be decomposed into a continuum of elementary blocks traveling with distinct velocities and that this effectively results into spreading and uniform decay\footnote{Those two are actually strongly connected by some form of preservation of $L^2$-norms.} of the solution. As expounded in Section~\ref{s:preparation}, in the periodic context the existence of a continuous decomposition follows from the Bloch-wave decomposition\footnote{See Subsection~\ref{s:Bloch} for definitions. At first reading, to follow the rest of the introduction it is essentially sufficient to know that there is \emph{some} continuous decomposition parametrized by $\xi$ and that phenomena that are slow up to periodic oscillations correspond to $|\xi|\ll 1$.}, where each block is parametrized by a Floquet exponent $\xi$, combined with a spectral decomposition of Bloch symbol $L_\xi$ that provides the action of $L$ on the $\xi$-part of the Floquet decomposition. Then we still need to know that velocities corresponding to the decomposition vary in a sufficiently non-degenerate way to yield dispersive spreading. This is the content of the following Assumption.
\be
\label{A_intro}
\begin{array}{l}
\textrm{At no point of spectral curves the second-order and}\\
\textrm{third-order derivatives with respect to Floquet exponents}\\ 
\textrm{vanish simultaneously.}
\end{array}\tag{A$_0$}
\ee
We give a more precise account of Assumption~\ref{A_intro} in Section~\ref{s:A}. Despite the relatively explicit description of the spectrum of $L$, the author has not been able to prove that this assumption does hold for any cnoidal wave. One may indeed prove in a rather straight-forward way that this condition holds in distinguished asymptotic regimes, that is, for large or for small eigenvalues. But out of these limits the fact that the explicit description does not provide directly a parametrization of the spectrum in terms of the Floquet exponent but rather a parametrization of both Floquet exponent and spectral parameter in terms of a third auxiliary variable, the spectral Lax parameter, leads any attempt to derive a closed-form for those derivatives to cumbersome expressions with signs not readily apparent. However based on easy well-conditioned numerics one may plot corresponding graphs. Experiments of the author --- see Appendix~\ref{s:A_numerics} --- based on eye inspection of those graphs clearly indicate that condition~\ref{A_intro} is always satisfied. 

\begin{theorem}{\emph{Asymptotic linear stability.}}\label{th:asymptotic}
For any cnoidal wave such that condition~\ref{A_intro} holds, there exists $C$ such that for any $W_0$ such that $N_{L^1(\R)}(W_0)<\infty$ and any time $t\in\R$ such that\footnote{One may remove the restriction $|t|\geq1$ under the stronger form of condition~\ref{A_intro} described below, see condition~\ref{A_intro-bis}. The weaker form considered here would instead provide $C\,\max(\{|t|^{-1/2},|t|^{-1/3}\}) N_{L^1(\R)}(W_0)$ as a global bound.} $|t|\geq1$
$$
N_{L^\infty(\R)}(S(t)\,W_0)\ \leq\ C\,|t|^{-1/3} N_{L^1(\R)}(W_0)\,.
$$
In particular for any such wave, there exists $C$ such that for any time $t\in\R$ and any $W_0$ such that $N_{L^1(\R)\cap H^1(\R)}(W_0)<\infty$
$$
N_{L^\infty(\R)}(S(t)\,W_0)\ \leq\ C\,(1+|t|)^{-1/3} N_{L^1(\R)\cap H^1(\R)}(W_0)\,.
$$
\end{theorem}

%Here, as follows for instance from the analysis in \cite{Bottman-Deconinck}, it turns out that for any cnoidal wave the spectrum of $L$ covers the full imaginary axis.

As already alluded to above, establishing decay here involves the proof of global-in-time dispersive estimates for operators with variable coefficients --- actually quite far from having constant coefficients. Needless to say that such form of results is quite unusual in the literature. The only related results the author is aware of are due to Cuccagna \cite{Cuccagna_dispersion_periodic-potential,Cuccagna_KG_periodic-1D} and Prill \cite{Prill_dispersive_KG_periodic-potential}, and have been subsequently used to prove nonlinear results with sufficiently nice nonlinearities \cite{Cuccagna-Visciglia_scattering_small-energy_periodic-potential,Prill_asymptotic-stability_vacuum_1D_NL-KG}. A significant difference is that periodicity in their cases stem from the presence of a periodic potential, and that they linearize about the zero solution hence receive a self-adjoint operator, which extends the range of techniques available.

\subsection{Slow modulation behavior}

Actually numerical experiments suggests that a stronger version of condition~\ref{A_intro} holds, namely 
\be
\label{A_intro-bis}
\begin{array}{l}
\textrm{At no nonzero point of spectral curves the second-order}\\
\textrm{derivative with respect to Floquet exponents vanish}\\ 
\textrm{and the third-order derivatives do not vanish at zero.}
\end{array}\tag{A}
\ee
The spectral point $0$ is associated with Floquet exponent $0$. Thus through classical considerations on oscillatory integrals condition~\ref{A_intro-bis} leads to the fact that critical decay rate $|t|^{-1/3}$ corresponds to the evolution of the spectral part of the initial data corresponding to small spectrum and small Floquet exponents, the remaining part of the solution decaying faster, at rate $|t|^{-1/2}$. It may then be expected that one could accurately describe the long-time evolution within a two-scale \emph{ansatz}, essentially a periodic oscillation when looking at a fixed bounded domain but whose characteristics evolve in time and space on larger scales. Moreover the well-known fact that eigen modes corresponding to the spectral point $0$ are given in terms of variations at $\uU$ along the manifold of periodic traveling wave profile suggests that in a formal \emph{ansatz} the local structure of oscillations could be captured by picking at each spatio-temporal point one neighboring periodic wave and slow evolutions would then result from a slow spatio-temporal motion along the manifold of periodic traveling waves. This is commonly referred to as a slow modulation behavior.

Our following results prove that this intuition is indeed correct at the linearized level we consider here and give a precise account of the large-time asymptotic behavior. To state it we first choose a parametrization of periodic traveling waves. Many choices are available in the present case, some of them being very explicit, others diagonalizing the first-order system formally driving at main order the slow evolution... We choose one of them that is not very explicit but both simple and known to be available essentially near any non degenerate wave of a system of partial differential equations. We refer the reader to \cite{Whitham,Kamchatnov,Benzoni-Mietka-Rodrigues_2} for a look at other possible parametrizations for the Korteweg--de Vries equation and to \cite[Appendix~B.2]{Benzoni-Noble-Rodrigues} or \cite[Section~2.1]{R} for a proof that our choice is still available in a broader context. In the context of the Korteweg--de Vries equation simple quadrature combined with reduction by symmetry shows that periodic traveling waves are smoothly given as $U(t,x)=\Up{k,M,P}(k\,x\,+\,\omega(k,M,P)\,t+\phi)$ with
$$
\Up{k,M,P}\textrm{ of period }1,\qquad
\int_0^1 \Up{k,M,P}\,=\,M,\qquad
\int_0^1 \tfrac12(\Up{k,M,P})^2\,=\,P\,.
$$
Correspondingly the phase velocity is given as $c(k,M,P)=-\omega(k,M,P)/k$. Note in particular that for some average values $\uM$, $\uP$ and some phase shift $\upsi$ we have $\uU=\Up{\uk,\uM,\uP}(\,\cdot\,+\upsi)$, $\uom=\omega(\uk,\uM,\uP)$ and $\uc=c(\uk,\uM,\uP)$. By translating profiles with $\upsi$ we may actually ensure $\upsi=0$. For the sake of writing convenience we shall do so from now on. In the following we shall also denote by $\bdD\Up{}$ the differential with respect to parameters $(k,M,P)$. Now the validation at our linearized level of the slow modulation scenario takes the following simple form.

\begin{theorem}{\emph{Slow modulation behavior.}}\label{th:behavior}
Assume that the cnoidal wave of parameters $(\uk,\uM,\uP)$ and phase shift zero is such that condition~\ref{A_intro-bis} holds. There exists $C$ such that for any $W_0$ such that $N_{L^1(\R)}(W_0)<\infty$ there exist %a global phase shift $\psiinf$ and 
local parameters $\psi$, $M$ and $P$ such that for any time $t\in\R$ with\footnote{One may relax the constraint $t\geq1$ and replace $|t|^{-1/2}$ with $\min(\{|t|^{-1/2},|t|^{-1/3}\})$. A similar remark applies to other similar restrictions below.} $|t|\geq1$
$$
\begin{array}{rcl}
\ds
\big\|S(t)\,(W_0)&-&\ds
\psi(t,\cdot)\,\Up{\uk,\uM,\uP}_x-\,\bdD\Up{\uk,\uM,\uP}\cdot(\uk\d_x\psi(t,\cdot),M(t,\cdot),P(t,\cdot))
\big\|_{L^\infty(\R)}\\[0.5em]
&\leq&C\,|t|^{-1/2}\,N_{L^1(\R)}(W_0)
\end{array}
$$
where $\psi(t,\cdot)$ is centered, $\psi(t,\cdot)$, $M(t,\cdot)$ and $P(t,\cdot)$ are low-frequency and 
$$
\|(\uk\d_x\psi(t,\cdot),M(t,\cdot),P(t,\cdot))\|_{L^\infty(\R)}
\,\leq\,C\,|t|^{-1/3}\,N_{L^1(\R)}(W_0)\,.
$$
\end{theorem}

We refer the reader to Section~\ref{s:slow-reduction} for precise definitions of the intuitive notions of being \emph{centered} or \emph{low-frequency}. A few other comments are in order.
\begin{enumerate}
\item The description implicitly contains the relation $k(t,\cdot)=\uk\d_x\phi(t,\cdot)$  between local wavenumber $k(t,\cdot)$ and local phase shift $\phi(t,\cdot)$ that is familiar in spatio-temporal modulation theories.
\item In the introduction we have chosen to state our results in a rather concise and  abstract form. Yet our proof shows that %if $W_0=\psi_0\,\uU_x+V_0$ with $((\psi_0)_x,V_0)\in L^1(\R)$ then 
$(\psi,M,P)$ may be chosen as explicit linear functions of $W_0$. %$(\psi_0,V_0)$.
\item The fact that $\psi(t,\cdot)$ is centered for any $t$ expresses that the time dynamics can not create any global-in-space phase-shift. This may seem in strong contrast with what happens near traveling waves with localized variations, such as solitary waves or fronts, where the main effect of perturbations is usually captured by a global-in-space phase-shift that evolves in time. The heuristics is as follows. In \emph{any} case perturbations affects solutions in a nearly local way. However for localized unimodal waves since at infinity the solution is approximately translation invariant and a single local shift effectively occurs the main effect may be described by a global-in-space phase-shift. 
\item Note carefully that $\psi$ does not decay so that even when $W_0$ is small one can not replace the first quantity estimated with a more nonlinear form
$$
\uU(x)+S(t)(W_0)(x)-\Up{\uk+\uk\psi_x(t,x),\uM+M(t,x),\uP+P(t,x)}(x+\psi(t,x))\,.
$$
This is consistent with the fact that the above form is not the right formula to undo the linearization expounded at the introduction of space-modulated norms. One correct formulation is that $S(t)(W_0)(x)=\psi(t,x)\,\uU_x(x)+V(t,x)$ and, with 
$$\Psi(t,\cdot)=(\Id_\R-\psi(t,\cdot))^{-1},\qquad W(t,x)=\uU(\Psi(t,x))+V(t,\Psi(t,x))$$ 
and
$$(\kappa,\cM,\cP)(t,x)=(\uk\d_x\Psi(t,x),\uM+M(t,\Psi(t,x)),\uP+P(t,\Psi(t,x)))\,,$$
the quantity
$$
W(t,x)-\Up{\kappa(t,x),\cM(t,x),\cP(t,x)}(\Psi(t,x))
$$
is bounded by a constant multiple of $|t|^{-1/2}\,N_{L^1(\R)}(W_0)$ provided that $N_{L^1(\R)}(W_0)$ is sufficiently small.
\end{enumerate}

\subsection{Averaged dynamics}

The last thing we would like to do is to provide a large-time description of the dynamics of $(\psi,M,P)$ introduced in the foregoing theorem. In other words we would like to identify some equivalent averaged dynamics. Relevant effective equations may indeed be obtained as a correction to the famous first-order Whitham system, linearized about the constant parameters $\ua=(\uk,\uM,\uP)$ of our reference background wave
\be\label{e:Wlin-3rd}
\begin{pmatrix}k\\M\\P\end{pmatrix}_t
+\uom\begin{pmatrix}k\\M\\P\end{pmatrix}_x
+\uk\begin{pmatrix}\dd\omega(\ua)\cdot(k,M,P)_x\\P_x\\\dd F(\ua)\cdot(k,M,P)_x\end{pmatrix}
\ =\ \uk^3\,D(\ua)\,\begin{pmatrix}k\\M\\P\end{pmatrix}_{xxx}\,
\ee
where $D(\ua)$ is a $3\times3$-matrix and
$$
F(k,M,P)
\,=\,\int_0^1\left(\tfrac13(\Up{(k,M,P)})^3-\tfrac32k^2((\Up{(k,M,P)})')^2\right)
$$
is the averaged of the flux associated with conservation law for Benjamin's impulse $\tfrac12 U^2$. A higher-order correction to the classical Whitham theory is required to describe accurately large-time behavior. Indeed one needs to reach a level of description accounting for dispersive effects, hence here the third order is the lowest relevant order of description. 

Such a suitable system could actually be derived by arguing on formal grounds and those formal derivations may be thought as geometric optic expansions of WKB type, following from a higher-order version of the two-timing method of Whitham in the spirit of \cite{Noble-Rodrigues}. However to keep the analysis as tight as possible we here follow a different process expounded below.

The gain on going from the full scalar linear equation $W_t-LW$ to system~\eqref{e:Wlin-3rd} is of averaging nature as $L$ has periodic coefficients while linear averaged systems are constant-coefficients systems and as such are expected to be much easier to understand directly. Yet as in classical homogenization problems the coefficients of reduced systems require averaging quantities depending on solutions of cell-problems. In particular $D(\ua)$ has a quite daunting explicit form. %TODO: add precise reference to where D is
It should be noted however that on one hand knowing that such reduction exists disregarding the specific form of the system already yields a wealth of information and that on the other hand if needed the coefficients involved may be computed numerically in a relatively simple way.

In any case the above-mentioned formal arguments do not yield any insight on effective initial data. Besides putting on sound mathematical grounds those formal arguments the main achievement of the following result is to provide equivalent initial data for averaged systems.

\begin{theorem}{\emph{Averaged dynamics, third order.}}\label{th:3rd}
Assume that the cnoidal wave of parameters $(\uk,\uM,\uP)$ and phase shift zero is such that condition~\ref{A_intro-bis} holds. There exists $C$ such that for any $W_0$ such that $N_{L^1(\R)}(W_0)<\infty$ there exists $\psi_0$ centered and low-frequency such that with $V_0=W_0-\uU_x\psi_0$ 
$$
\|V_0\|_{L^1(\R)}\,+\,\|\d_x\psi_0\|_{L^1(\R)}
\,\leq\,2\,N_{L^1(\R)}(W_0)
$$
and for any such $\psi_0$ the local parameters $(\psi,M,P)$  of Theorem~\ref{th:behavior} may be chosen in such a way that for any time $t\in\R$ with $|t|\geq1$
$$
%\begin{array}{rcl}
%\ds
\big\|(\uk\d_x\psi(t,\cdot),M(t,\cdot),P(t,\cdot))-%&-&\ds
\SigW(t)(\uk\d_x\psi_0,V_0,\uU\,V_0)
\big\|_{L^\infty(\R)}%\\[0.5em]
%&
\,\leq\,%&
C\,|t|^{-2/5}\,N_{L^1(\R)}(W_0)
%\end{array}
$$
and 
$$
\|\psi(t,\cdot)-\eD_1\cdot\,(\uk\d_x)^{-1}\SigW(t)(\uk\d_x\psi_0,V_0,\uU\,V_0)\|_{L^\infty(\R)}
\,\leq\,C\,|t|^{-1/5}\,N_{L^1(\R)}(W_0)
$$
where $\SigW$ denotes the solution operator for System~\ref{e:Wlin-3rd}.
\end{theorem}

Of course the point is that $|t|^{-2/5}$ is negligible in front of $|t|^{-1/3}$ in the large-time limit. Note also that while $\psi$ does not decay to zero in the large-time limit we do provide a description of $\psi$ up to eventually vanishing terms. This is a crucial achievement since creating phase shifts is indeed the leading effect of perturbations. It is important\footnote{But classical, even for asymptotically constant waves.} however to understand that whereas knowing phase shifts is in principle sufficient to construct a leading-order description of the original solution, the obtention of the the dynamical behavior of the phase itself requires a knowledge of all the modulation parameters. In particular even when one may enforce $\psi_0\equiv0$ the time evolution will still create a significant phase shift.

The fact that in our statement the prescription of effective initial data for modulation equations is relatively simple in terms of $(V_0,\psi_0)$  is closely related to the fact that we \emph{choose} $\psi_0$ to be low-frequency, which is consistent with a \emph{slow} modulation scenario. It is actually possible to pick any $(V_0,\psi_0)$ such that $W_0=V_0+\uU_x\psi_0$ and obtain equivalent statement where $\|V_0\|_{L^1(\R)}\,+\,\|\d_x\psi_0\|_{L^1(\R)}$ replaces $N_{L^1(\R)}(W_0)$ but then effective initial data have a more complicated form that encode projection to \emph{slow} phase shift. Explicitly in this case, in Theorems~\ref{th:3rd} and~\ref{th:qth}, $(\uk\d_x\psi_0,V_0,\uU\,V_0)$ should be replaced with 
\be\label{hf-data}
\begin{pmatrix}\uk\d_x\psi_0\\\ds
V_0-\left(\uU-\int_0^1\uU\right)\,\d_x\psi_0\\\ds
\uU\,V_0-\left(\tfrac12\uU^2-\int_0^1\tfrac12\uU^2\right)\,\d_x\psi_0\end{pmatrix}\,.
\ee
See \cite[Remark~1.14]{JNRZ-conservation} for a more detailed, related discussion.

In view of the decay rates obtained in Theorem~\ref{th:behavior} and the heuristics concerning orders of vanishing derivatives of spectral curves one may rightfully wonder whether there is a way to obtain a more precise description achieving $\cO(|t|^{-1/2})$ remainders. It is indeed possible to reach this precision if one replaces the third order correction with a pseudo-differential one. Moreover one may achieve rates intermediate between $\cO(|t|^{-2/5})$ and $\cO(|t|^{-1/2})$ infinitely close to $\cO(|t|^{-1/2})$ by replacing the third order correction with higher-order differential corrections. This is the content of our last main results.

However it seems hard to obtain those higher-order corrections by formal arguments of geometric optic type. Instead the higher-order systems may be obtained directly in a way that we explain now. We first make the following observations, to be obtained as corollaries of the proofs of our main results, that the first-order system
\be\label{e:Wlin-1st}
\begin{pmatrix}k\\M\\P\end{pmatrix}_t
+\uom\begin{pmatrix}k\\M\\P\end{pmatrix}_x
+\uk\begin{pmatrix}\dd\omega(\ua)\cdot(k,M,P)_x\\P_x\\\dd F(\ua)\cdot(k,M,P)_x\end{pmatrix}
\ =\ 0\,
\ee
is strictly hyperbolic and that when diagonalizing the corresponding operator as $Q_0^{-1}\textrm{diag}(\d_t+a_0^{(0)}\d_x,\d_t+a_0^{(1)}\d_x,\d_t+a_0^{(2)}\d_x)Q_0^{-1}$ one obtains first-order expansions of the three Floquet eigenvalues $\lambda_0(\xi)$, $\lambda_1(\xi)$, $\lambda_2(\xi)$ passing trough the origin, $\lambda_j(\xi)=\iD\xi a_0^{(j)}+\grandO(|\xi|^3)$ as $\xi\to0$. Now we claim that it is sufficient to include dispersion corrections 
\be\label{e:Wlin-qth}
\begin{pmatrix}k\\M\\P\end{pmatrix}_t
+\uom\begin{pmatrix}k\\M\\P\end{pmatrix}_x
+\uk\begin{pmatrix}\dd\omega(\ua)\cdot(k,M,P)_x\\P_x\\\dd F(\ua)\cdot(k,M,P)_x\end{pmatrix}
\ =\ D_q(\iD^{-1}\d_x)\begin{pmatrix}k\\M\\P\end{pmatrix}\,
\ee
through 
$$
D_q(\xi)=Q_0^{-1}\textrm{diag}(\lambda_0^{(q)}(\xi)-a_0^{(0)}\iD\xi,\lambda_1^{(q)}(\xi)-a_0^{(1)}\iD\xi,\lambda_2^{(q)}(\xi)-a_0^{(2)}\iD\xi)Q_0^{-1}
$$
where $\lambda_j^{(q)}(\xi)$ is the $q$th order Taylor expansion of $\lambda_j(\xi)$ near $0$. By convention we also include the pseudo-differential case where $q=\infty$ by choosing $\lambda_j^{(\infty)}$ as a smooth real-valued function that coincide with $\lambda_j$ in a neighborhood of zero. For simplicity, in \eqref{e:Wlin-3rd}, we have also chosen $D=D_3$.

\begin{theorem}{\emph{Averaged dynamics, higher order.}}\label{th:qth}
Assume that the cnoidal wave of parameters $(\uk,\uM,\uP)$ and phase shift zero is such that condition~\ref{A_intro-bis} holds.\\
Let $q$ be an odd integer larger than $3$, or $q=\infty$.

There exists $C$ and a cut-off function $\chi$ such that for any $W_0$ such that $N_{L^1(\R)}(W_0)<\infty$ there exists $\psi_0$ centered and low-frequency such that with $V_0=W_0-\uU_x\psi_0$ 
$$
\|V_0\|_{L^1(\R)}\,+\,\|\d_x\psi_0\|_{L^1(\R)}
\,\leq\,2\,N_{L^1(\R)}(W_0)
$$
and for any such $\psi_0$ the local parameters $(\psi,M,P)$  of Theorem~\ref{th:behavior} may be chosen in such a way that for any time $t\in\R$ with $|t|\geq1$
$$
%\begin{array}{rcl}
%\ds
\big\|(\uk\d_x\psi(t,\cdot),M(t,\cdot),P(t,\cdot))-%&-&\ds
\SigW_q(t)\chi(\iD^{-1}\d_x)(\uk\d_x\psi_0,V_0,\uU\,V_0)
\big\|_{L^\infty(\R)}%\\[0.5em]
%&
\,\leq\,%&
C\,|t|^{-(q+1)/(2(q+2))}\,N_{L^1(\R)}(W_0)
%\end{array}
$$
and 
$$
\|\psi(t,\cdot)-\eD_1\cdot\,(\uk\d_x)^{-1}\SigW_q(t)\chi(\iD^{-1}\d_x)(\uk\d_x\psi_0,V_0,\uU\,V_0)\|_{L^\infty(\R)}
\,\leq\,C\,|t|^{-1/3}\,N_{L^1(\R)}(W_0)
$$
where $\SigW_q$ is the solution operator to System~\ref{e:Wlin-qth}.
\end{theorem}

The foregoing construction of $D_q$ follows closely the classical construction of artificial viscosity system as large-time asymptotic equivalents to systems that are only parabolic in the hypocoercive sense of Kawashima. We refer the reader for instance to \cite[Section~6]{Hoff_Zumbrun-NS_compressible_pres_de_zero}, \cite{Rodrigues-compressible}, \cite[Appendix~B]{JNRZ-conservation} or \cite[Appendix~A]{R} for a description of the latter. A notable difference however is that in the diffusive context higher-order expansions of dispersion relations beyond the second-order necessary to capture some dissipation does not provide any sharper description as the second-order expansion already provides the maximal rate compatible with a first-order expansion of eigenvectors. 

We stress also that a significant difference with the third-order case dealt with in Theorem~\ref{th:3rd} is the necessity to introduce the low-frequency cut-off $\chi(\iD^{-1}\d_x)$. This is due to the fact that for higher-order expansions one can not derive good dispersion properties for the full evolution from the mere knowledge of such behavior for the low-frequency part. This is somehow analogous to the fact that slow expansions of well-behaved parabolic systems may produce ill-posed systems.

At last one may also improve the description of the phase itself up to $\cO(|t|^{-1/2})$ remainders. But this requires a suitably tailored refined effective initial data.

\begin{theorem}{\emph{Averaged dynamics, sharpest description.}}\label{th:qth-phase}
Assume that the cnoidal wave of parameters $(\uk,\uM,\uP)$ and phase shift zero is such that condition~\ref{A_intro-bis} holds.\\
Let $q$ be an odd integer larger than $3$, or $q=\infty$.

There exists $C$ such that for any $W_0$ such that $N_{L^1(\R)}(W_0)<\infty$ there exists a low-frequency $(\tpsi_0,\tM_0,\tP_0)$ such that the local parameters $(\psi,M,P)$  of Theorem~\ref{th:behavior} may be chosen in such a way that for any time $t\in\R$ with $|t|\geq1$
$$
%\begin{array}{rcl}
%\ds
\big\|(\uk\d_x\psi(t,\cdot),M(t,\cdot),P(t,\cdot))-%&-&\ds
\SigW_q(t)(\tpsi_0,\tM_0,\tP_0)
\big\|_{L^\infty(\R)}%\\[0.5em]
%&
\,\leq\,%&
C\,|t|^{-(q+1)/(2(q+2))}\,N_{L^1(\R)}(W_0)
%\end{array}
$$
and 
$$
\|\psi(t,\cdot)-\eD_1\cdot\,(\uk\d_x)^{-1}\SigW_q(t)(\tpsi_0,\tM_0,\tP_0)\|_{L^\infty(\R)}
\,\leq\,C\,|t|^{-(q-1)/(2(q+2))}\,N_{L^1(\R)}(W_0)\,.
$$
\end{theorem}

Before entering into proofs of our main statements, to make those statements slightly more concrete let us summarize what we have learned \emph{at leading order} from Theorems~\ref{th:behavior} and~\ref{th:3rd}. At leading order the behavior of $S(t)(W_0)$ is captured by a linear modulation of phase $\psi(t,\cdot)\,\uU_x$ and the phase shift $\uk\psi$ is the antiderivative of the first component of a three-dimensional vector $(\uk\psi_x,M,P)$ that is at leading-order a sum of three linear dispersive waves of Airy type, each one traveling with its own velocity. In particular, three scales coexist : the oscillation of the background wave at scale $1$ in $\uU_x$, spatial separation of the three dispersive waves at linear hyperbolic scale $t$, width of Airy waves of size $t^{1/3}$. This is illustrated by direct simulations in Figure~\ref{fig:multiscale}. To fully appreciate the figure, note that oscillatory Airy tail is on the left for the left-hand side and right-hand side dispersive waves and on the right for the middle one.

\begin{figure}
\begin{center}
$
\begin{array}{lcr}\hspace{-2.5em}
\includegraphics[scale=0.32]{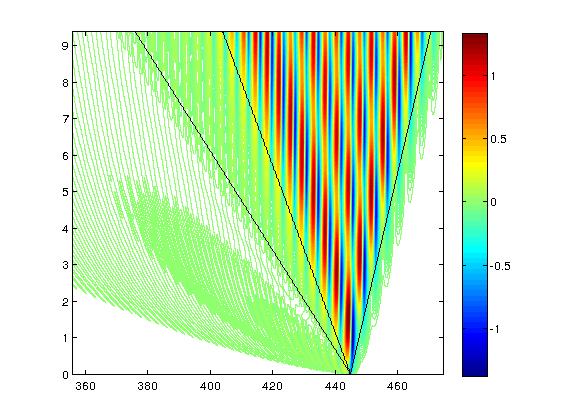}\hspace{-2.5em}
&\includegraphics[scale=0.32]{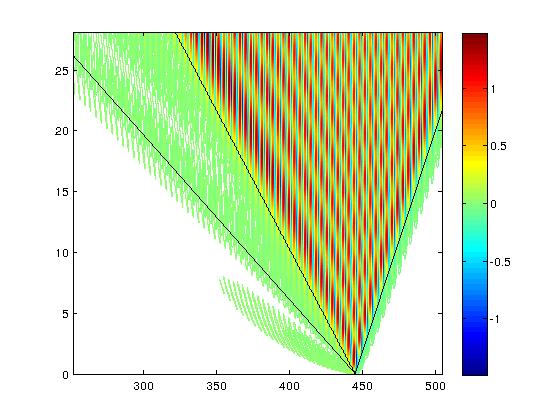}\hspace{-2.5em}
&\includegraphics[scale=0.32]{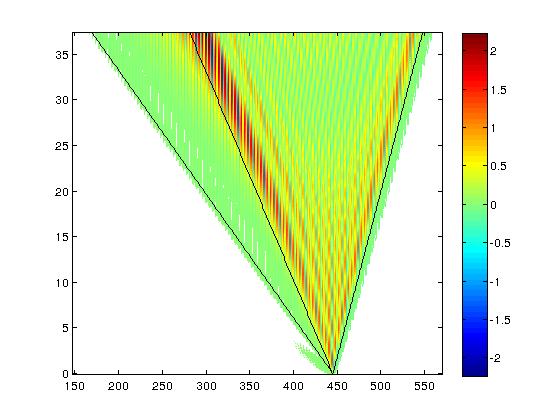}
\end{array}
$
\end{center}
\caption{Three looks at the same time-evolution. The background wave has elliptic parameter $m=0.5$ hence period approximately $3.7081$. Initial data for the perturbation $W_0$ is the product of a sinus with a Gaussian. Dark lines start from the center of the Gaussian, and corresponds to linear group velocities.}
\label{fig:multiscale}
\end{figure}

\subsection{Perspectives}

For the linearized Korteweg--de Vries equation itself, besides the question of proving condition~\ref{A_intro-bis}, still remains the question of providing a derivation of suitable modulation systems similar to \eqref{e:Wlin-qth}, when $q>3$, by formal arguments, either by using directly a geometric optic \emph{ansatz} or by expanding the Hamiltonian energy.

Recall also that the decay proved here is too slow to be directly relevant at nonlinear level. From this, two natural follow-up questions arise:
\begin{enumerate}
\item At the nonlinear level, for the Korteweg--de Vies equation, can we still provide a --- more nonlinear! --- slow modulation description of the asymptotic behavior obtained in \cite{Mikikits-PhD,Mikikits-Teschl} ?
\item Can we perform a similar linear analysis in another situation that could be carried to the nonlinear level ?
\end{enumerate}
On the latter, natural candidates are to be found in dynamics near periodic plane waves of dispersive systems in sufficienty high dimension.

\subsection{Structure of the paper}

The remaining part of the paper is devoted to the proofs of foregoing theorems. In the next section we first recall some elements of Bloch analysis, extract from \cite{Bottman-Deconinck} detailed information on the spectrum of $L$ and derive from it some representations of the corresponding time-evolution. In particular we provide both a spectral decomposition of the evolution and its counterpart in terms of Green kernels. We also gather there spectral asymptotic expansions in singular limits where either Floquet eigenvalues converge to zero or go to $\pm\iD\infty$. In the third section we prove Theorems~\ref{th:bounded} and~\ref{th:asymptotic}, by using respectively the above-mentioned spectral and kernel representations. In the fourth section we achieve the proofs of remaining results. Those rely mostly on low Floquet/low eigenvalue expansions in the spirit of \cite{JNRZ-conservation} combined with suitable oscillatory integral estimates. Proofs of the latter are given in Appendix~\ref{s:oscillatory}. We point out that though the subject is quite classical Appendix~\ref{s:oscillatory}, oriented towards derivation of asymptotically equivalent systems, could be of some general interest. In Appendix~\ref{s:A_numerics} we gather some numerical experiments supporting that Assumption~\ref{A_intro-bis} always holds.

%%%%%%%%%%%%%%%%%%%%%%%%%%%%%%%%%%%%%%%%%%%%%%%%%%%%%% 
%        SPECTRUM                                    %
%%%%%%%%%%%%%%%%%%%%%%%%%%%%%%%%%%%%%%%%%%%%%%%%%%%%%%

\section{Spectral preparation}\label{s:preparation}

\subsection{Integral transform}\label{s:Bloch}

We first recall how to decompose any function $g$ into a sum of functions that are simpler from the point of view of periodicity, namely
\be\label{inverse-Bloch}
g(x)\ =\ \int_{-\pi}^\pi \eD^{\ii\xi x}\ \check{g}(\xi,x)\ \dD\xi,
\ee
with each $\check{g}(\xi,\cdot)$ periodic of period one, that is
$$
\forall x\in\R,\qquad\check{g}(\xi,x+1)\ =\ \check{g}(\xi,x).
$$

Such an inverse formula may be obtained by rewriting appropriately an inverse Fourier decomposition. For this purpose we introduce direct and inverse Fourier transforms, \emph{via}
$$
\hat g(\xi)\ :=\ \frac{1}{2\pi} \int_\R \eD^{-\ii\xi x} g(x)\ \dD x,\qquad
g(x)\ =\ \int_\R \eD^{\ii\xi x}\ \hat{g}(\xi)\ \dd\xi.
$$
Then the adequate integral transform, called the \emph{Bloch transform} or the Floquet-Bloch transform, may be defined by
\be\label{Bloch}
\cB(g)(\xi,x)\ =\ \check{g}(\xi,x)\ :=\ \sum_{j\in\Z} \eD^{\ii\,2j\pi x}\ \widehat{g}(\xi+2j\pi).
\ee
The Poisson summation formula provides an alternative equivalent formula
$$
\check{g}(\xi,x)\ =\ \sum_{\ell\in\Z}\eD^{-\ii\xi (x+\ell)}g(x+\ell)\,.
$$
As follows readily from \eqref{Bloch}, $\sqrt{2\pi}\,\cB$ is a total isometry from $L^2(\R)$ to $L^2((-\pi,\pi),L_{per}^2((0,1)))$. Interpolating with triangle inequalities also yields Hausdorff-Young inequalities, for $2\leq p\leq\infty$,
$$
\|g\|_{L^p(\R)}\ \leq\ (2\pi)^{1/p}
\| \check{g}\|_{L^{p'}([-\pi,\pi],L^p((0,1)))}\,,\qquad
\|\check{g}\|_{L^{p}([-\pi,\pi],L^{p'}((0,1)))}\ \leq
\ (2\pi)^{-1/p}\|g\|_{L^{p'}(\R)}\,
$$
where $p'$ denotes conjugate Lebesgue exponent, $1/p+1/p'=1$.

We have introduced the Bloch transform so as to turn differential operators with periodic coefficients in multipliers with respect to the Floquet exponent $\xi$. Indeed for $L$ as in \eqref{linop} we have
$$
(Lg)(x)\ =\ \int_{-\pi}^\pi \eD^{\ii\xi x}\ (L_\xi\check{g}(\xi,\cdot))(x)\ \dD\xi,
$$
where each $L_\xi$ acts on periodic functions as
$$
L_\xi\ :=\ -\,\uom\,(\d_x+\ii\xi)\ -\ \uk\,(\d_x+\ii\xi)(\uU\,\cdot)\ -\uk^3\,(\d_x+\ii\xi)^3\,.
$$
On $L_{per}^2((0,1))$ each $L_\xi$ has compact resolvent and it depends analytically on $\xi$ in the strong resolvent sense.

\subsection{Spectrum of $L$}\label{s:spec}

Now we recall the content of \cite[Theorem~7.1]{Bottman-Deconinck}, slightly extended by using \cite[Remark~4]{Bottman-Deconinck} and some extra functional-analytic arguments.

We have fixed a cnoidal wave profile $\uU$ to \eqref{KdV}. Though such waves form a four-dimensional family one may use Galilean invariance and invariances by spacial translation and a suitable scaling to restrict the present discussion to a one-dimensional sub-family
$$
\uU(x)\ =\ 12\,m\,\cn^2(\,\tfrac{x}{\uk}\,,m)
$$
where $m$ is an elliptic parameter\footnote{Note carefully that it is the square of the elliptic modulus.}, $m\in(0,1)$, $\cn(\,\cdot\,,m)$ denotes the corresponding Jacobi elliptic cosine function and wavenumber $\uk$ is such that
$$
\frac{2}{\uk}\ =\ \int_0^{\pi/2}\,\dfrac{\dd\theta}{\sqrt{1-m\sin^2(\theta)}}\,.
$$ 
Corresponding velocity is then given by $\uc=4\,(2m-1)$.

For $\uU$ as above we set $\eta_1=4\,(m-1)$, $\eta_2=4\,(2m-1)$ and $\eta_3=4m$. Then for any couple $(\lambda,\xi)\in \C\times[-\pi,\pi)\setminus\{(0,0)\}$, $\lambda\in\sigma(L_\xi)$ if and only if there exists $\eta\in\,]-\infty,\eta_1)\,\cup\,(\eta_2,\eta_3)$ such that
$$
\lambda^2\ =\ (\eta-\eta_1)\,(\eta-\eta_2)\,(\eta-\eta_3)
$$
and
$$
\textrm{Im}\left(-\,\frac{\lambda}{\uk}\,\int_0^1\dfrac{\dd \theta}{\eta-\uc+\tfrac13\uU(\theta)}\right)\in \xi\ +\ 2\pi\,\Z\,.
$$
Moreover in this case $\lambda$ is a simple eigenvalue of $L_\xi$ and an eigenfunction $\phi$ is provided by
$$
\eD^{\ii\xi\,x}\phi(x)\ =\ \left(1-\tfrac{\uk}{3}\lambda^{-1}\uU_x(x)\right)\ \exp\left(\frac{-\lambda}{\uk}\int_0^x\dfrac{\dd \theta}{\eta-\uc+\tfrac13\uU(\theta)}\right)
$$
while a solution $\tphi$ of the formally dual problem
$$
\bar{\lambda}\,\tphi=\,\uom\,(\d_x+\ii\xi)\,\tphi\ +\ \uk\,(\d_x+\ii\xi)\left(\uU\,\tphi\right)\ +\uk^3\,(\d_x+\ii\xi)^3\,\tphi
$$
is given by 
$$
\eD^{\ii\xi\,x}\tphi(x)\ =\ \dfrac{\eta-\uc+\tfrac13 \uU(x)}{\eta-\uc+\tfrac13\int_0^1\uU}\ \exp\left(\frac{-\lambda}{\uk}\int_0^x\dfrac{\dd \theta}{\eta-\uc+\tfrac13\uU(\theta)}\right)\,.
$$
Normalization of $\phi$ and $\tphi$ ensures all together a suitable form of bi-orthogonality, detailed below, convergence to trigonometric monomials in the limit $|\lambda|\to\infty$ and the absence of singularities on $\tphi$ in the limit $\lambda\to0$.

To carry out our Floquet analysis we shall use some consistent labeling of the spectrum of each $L_\xi$. To this purpose we first observe that on $]-\infty,\eta_1)$ both
$$
\eta\mapsto\sqrt{(\eta_1-\eta)\,(\eta_2-\eta)\,(\eta_3-\eta)}
\quad\textrm{and}\quad
\eta\mapsto\frac{1}{\uk}\,\int_0^1\dfrac{\sqrt{(\eta_1-\eta)\,(\eta_2-\eta)\,(\eta_3-\eta)}}{\uc-\eta-\tfrac13\uU(\theta)}\,\dd \theta
$$
are decreasing, respectively from $\infty$ to $0$ and from $\infty$ to $2\pi$. Therefore we may parametrize the part of the spectrum of $L$ arising from $\eta\in\,]-\infty,\eta_1]$ as $\lambdae{j}(\xi)$, $(j,\xi)\in\Z\times[-\pi,\pi)$ in a way that ensures $\lambdae{0}(0)=0$; the map $(j,\xi)\mapsto\textrm{Im}(\lambdae{j}(\xi))$ is increasing when $\Z\times[-\pi,\pi)$ is endowed with alphabetical order, and odd; for any $(j,\xi)$, $\lambdae{j}(\xi)\in\sigma(L_\xi)$; and for any $j$, $\lambdae{j}(\xi)\stackrel{{\tiny\xi\to\pi}}{\longrightarrow}\lambdae{j+1}(-\pi)$. The structure of the spectrum related to $(\eta_2,\eta_3)$ is less obviously read on above formulas even though some pieces of information may be deduced from a count of multiplicity\footnote{For instance one knows in advance that each $\lambda$ is at most triply covered by $\cup_\xi\sigma(L_\xi)$. Hence the loop may only visit twice each $\lambda$.}. Since it is immaterial to our analysis, for simplicity of notation, we shall do as it could be minimally\footnote{Otherwise one would need to introduce a larger number of pieces parametrized by a finite number of $j$. Numerical experiments suggests that minimal parametrization does hold. This claim would follow from an examination of limits $\eta\to\eta_2$ and $\eta\to\eta_3$ provided we were able to prove monotonicity of 
$$
\eta\mapsto-\frac{1}{\uk}\,\int_0^1\dfrac{\sqrt{(\eta_1-\eta)\,(\eta_2-\eta)\,(\eta_3-\eta)}}{\uc-\eta-\tfrac13\uU(\theta)}\,\dd \theta
$$
on $(\eta_2,\eta_3)$ or provided that we were able to prove that the previous mapping does note take the value $-2\pi$ on $(\eta_2,\eta_3)$.} parametrized, that is we shall write it as\footnote{The notational convention ${}_{\rm p}$ and ${}_{\rm e}$ is motivated by the fact that in the large period limit the spectrum associated with the line may be thought as arising from the essential spectrum of the generator of the dynamics linearized about a solitary wave while the loop emerges from an embedded eigenvalue. See \cite{Gardner-large-period,Sandstede-Scheel-large-period}.} $\lambdap{j}(\xi)$, $(j,\xi)\in\{1,2\}\times[-\pi,\pi)$ in a way that ensures $\lambdap{j}(0)=0$ for any $j$; each map $\xi\mapsto\textrm{Im}(\lambdap{j}(\xi))$ is odd; for any $(\sigma,\xi)$, $\lambdap{j}(\xi)\in\sigma(L_\xi)$; and $\lambdap{1}(\xi)\stackrel{{\tiny\xi\to\pi}}{\longrightarrow}\lambdap{2}(-\pi)$, $\lambdap{2}(\xi)\stackrel{{\tiny\xi\to\pi}}{\longrightarrow}\lambdap{1}(-\pi)$. Likewise we shall use notation $\phie{j}(\xi,x)$, $\phip{j}(\xi,x)$, $\tphie{j}(\xi,x)$, $\tphip{j}(\xi,x)$, for corresponding eigenfunctions.

We use the above spectral decomposition to represent when $\xi\neq0$ the evolution generated by $L_\xi$ as 
$$
S_\xi(t)\ =\ \sum_{j\in\{1,2\}} \eD^{\lambdap{j}(\xi)\,t}\ \phip{j}(\xi,\cdot)\,\langle \tphip{j}(\xi,\cdot),\cdot\rangle
\ +\ \sum_{j\in\Z} \eD^{\lambdae{j}(\xi)\,t}\ \phie{j}(\xi,\cdot)\,\langle \tphie{j}(\xi,\cdot);\cdot\rangle
$$
where $\langle\cdot;\cdot\rangle$ denotes canonical Hermitian scalar product\footnote{That is, $\langle f;g \rangle=\int_0^1 \bar f g$.} on $L^2((0,1))$. Even when $t=0$ this requires some justification. Let us first observe that since for each fixed Floquet exponent $\xi$ the operator $\uk^3(\d_x+\iD\xi)^3$ is skew-adjoint on $L_{per}^2((0,1))$ with trace-class resolvents\footnote{They belong to the Schatten class $\mathfrak{S}_p(L_{per}^2((0,1)))$ whenever $p>1/3$.}, 
and $L_\xi$ is a relatively compact perturbation of this operator, it follows from Keldysh' theory that root vectors of $L_\xi$ --- that is, $(\phie{j}(\xi,\cdot))_{j\in\Z}$ and $(\phip{j}(\xi,\cdot))_{j\in\{1,2\}}$ when $\xi\neq0$ --- form a complete subset of $L_{per}^2((0,1))$ \cite[Theorem~4.3]{Markus}. The same is true for root vectors of their formally adjoint operators. Since we have built simultaneously bi-orthogonal families, this ensures that they form minimally complete families. It actually follows from very general arguments that for $\xi\neq0$ they do form a Schauder basis \cite[Lemma~2.3]{Haragus-Kapitula-Hamiltonian}. Here we rather provide a direct proof of the fact that they form a Risez basis, that is, an \emph{unconditional} basis. This is a direct consequence of the characterization in \cite[Theorem~3.4.5]{Davies} together with estimates proving Theorem~\ref{th:bounded} when $s=0$, see Proposition~\ref{prop:bounded}.

\subsection{Small-eigenvalue expansions}\label{s:side-band}

When $\xi\to0$, $\lambdae{0}(\xi)$, $\lambdap{1}(\xi)$ and $\lambdap{2}(\xi)$ converge to zero and $\phie{0}(\xi,\cdot)$, $\phip{1}(\xi,\cdot)$ and $\phip{2}(\xi,\cdot)$ become singular\footnote{We have indeed normalized $\phi$ and $\tphi$ to ensure that singularities remain confined to \emph{right} eigenfunctions.}. However as follows from Kato's perturbation theory their combined evolution
$$
\eD^{\lambdae{0}(\xi)\,t}\ \phie{0}(\xi,\cdot)\,\langle \tphie{0}(\xi,\cdot);\cdot\rangle
\ +\ 
\eD^{\lambdap{1}(\xi)\,t}\ \phip{1}(\xi,\cdot)\,\langle \tphip{1}(\xi,\cdot);\cdot\rangle
\ +\ 
\eD^{\lambdap{2}(\xi)\,t}\ \phip{2}(\xi,\cdot)\,\langle \tphip{2}(\xi,\cdot);\cdot\rangle
$$
remains analytic in $\xi$ even when $\xi\to0$. Moreover one may readily check on explicit formulas that singularities are mild compared to those arising from generic splitting of Jordan block structures, that include algebraic singularities described by Puiseux series, see \cite[Section II.\S 1.2, p.65]{Kato}. In the case under study, since the spectrum lies on the imaginary axis such strong singularities are precluded by arguments similar to those leading to Rellich's theorem, see \cite[Theorem~XII.3]{Reed-Simon_IV}. However in the spectral analysis of dynamics linearized about periodic waves this fact turns out to be a much wider phenomenon that is strongly connected with the existence of an averaged dynamics, even when underlying waves are spectrally unstable and a Rellich-type argument fails. See for instance \cite{Serre,Noble-Rodrigues,Johnson-Zumbrun-Bronski-Whitham-KdV-Bloch,Benzoni-Noble-Rodrigues,KR}. In particular the small eigenvalue asymptotics provided below is a corollary of the proof of \cite[Theorem~1]{Benzoni-Noble-Rodrigues}.
%TODO: add something about "also included in the proof of the asymptotic behavior hence postponed" once those stated

To stress symmetries in its statement and in similar following propositions, we simply drop suffixes ${}_{\rm{p}}$ and ${}_{\rm{e}}$ when dealing with $\lambdae{0}$, $\lambdap{1}$ and $\lambdap{2}$, and $\tphie{0}$, $\tphip{1}$ and $\tphip{2}$. Concerning right eigenfunctions however our convention is that suffix-less functions are desingularized, that is for $\xi\neq 0$
\be
\phi_0(\xi,\cdot)\ =\ \iD\uk\xi\,\phie{0}(\xi,\cdot)
\qquad\textrm{and}\qquad
\phi_j(\xi,\cdot)\ =\ \iD\uk\xi\,\phip{j}(\xi,\cdot)\,,\ j=1,2\,.
\ee

\begin{proposition}\label{prop:slow-spec-expansion}
There exist $\eps_0>0$ and $\xi_0\in(0,\pi)$ such that curves $\lambda_j:[-\xi_0,\xi_0]\to B(0,\eps_0)$, $j=0,1,2$, are analytic and that for $\xi\in[-\xi_0,\xi_0]$
$$
\sigma(L_\xi)\cap B(0,\eps_0)\ =\ \left\{\ \lambda_j(\xi)\ \middle|\ j\in\{0,1,2\}\ \right\}
$$
and associated left and right eigenfunctions $\tphi_j(\xi,\cdot)$ and $\phi_j(\xi,\cdot)$, $j=0,1,2$, satisfying pairing relations
%\be\label{orthogg}
$$
\langle\tphi_j(\xi,\cdot),\phi_k(\xi,\cdot)\rangle=\iD\uk\xi\ \delta^j_k,\qquad 0\le j,k\le 2,
$$
%\ee
are obtained as
\be%\label{phibifg}
\begin{array}{rcccl}
\displaystyle
\phi_j(\xi,\cdot)&=&\displaystyle
\quad\beta_{0}^{(j)}(\xi)\quad q_{0}(\xi,\cdot)
&+&\displaystyle
(\iD\uk\xi)\sum_{l=1}^2\beta_l^{(j)}(\xi)\ q_l(\xi,\cdot)\\
\displaystyle
\tphi_j(\xi,\cdot)&=&\displaystyle
-(\iD\uk\xi)\,\tbeta_{0}^{(j)}(\xi)\quad\tq_{0}(\xi,\cdot)
&+&\displaystyle
\quad\sum_{l=1}^2\tbeta_l^{(j)}(\xi,\cdot)\ \tq_l(\xi,\cdot)\\
\end{array}
\ee
where
\begin{itemize}
\item $(q_0(\xi,\cdot),q_{1}(\xi,\cdot),q_{2}(\xi,\cdot))$ and $(\tq_0(\xi,\cdot),\tq_1(\xi,\cdot),\tq_{2}(\xi,\cdot))$ are dual bases of spaces associated with the spectrum in $B(0,\eps_0)$ of respectively $L_\xi$ and its formal adjoint $L_\xi^*$, that are analytic in $\xi$ and emerge from $(\,\ubU',\,*\,,\,*\,)$ and $(\,*\,,\one,\ubU\,)$ at $\xi=0$, where $\one$ denotes the constant function with value $1$;
\item $(\beta^{(0)}(\xi),\beta^{(1)}(\xi),\beta^{(2)}(\xi))$ and $(\tbeta^{(0)}(\xi),\tbeta^{(1)}(\xi),\tbeta^{(2)}(\xi))$ are dual bases of $\bC^3$ that are analytic in $\xi$.
\end{itemize}
\end{proposition}

%TODO: say a one word to explain mysterious roles.

The foregoing proposition enables us to perform a splitting of the evolution semi-group tailored to quantify space-modulated stability, and essentially identical to those used by the author and its collaborators in the study of parabolic systems, see \cite{JNRZ-conservation}. Explicitly,  
\be\label{e:spp+tS}
S(t)\ =\ \uU_x\ \eD_1\cdot\spp(t)\ +\ \tS(t)
\ee
with $\spp(t)=\sum_{j\in\{0,1,2\}}\spp_j(t)$, the action of $\spp_j(t)$, $j=0,1,2$, on a function $g$ being defined on the Fourier side by, for\footnote{With usual meaning that zero times something undefined equals zero.} $\xi\in\R$,
\be\label{e:spp}
\widehat{\spp_j(t)(g)}(\xi)\ =\ \dfrac{\chi(\xi)}{\iD\uk\xi}\,\beta^{(j)}(\xi)\ \eD^{\lambda_{j}(\xi)\,t}\left\langle\tphi_{j}(\xi,\cdot);\check{g}(\xi,\cdot)\right\rangle
\ee
and $\tS(t)=\tSo(t)+\tSe(t)$, $\tSo(t)$ and $\tSe(t)$ being defined by their Bloch symbols as, for $\xi\in(-\pi,\pi)$,
\be\label{e:tS}
\begin{array}{rcl}
(\tSo(t))_\xi
&=&\displaystyle 
\frac{(1-\chi(\xi))}{\iD\uk\xi}\ \sum_{j\in\{0,1,2\}}\ \eD^{\lambda_{j}(\xi)\,t}\phi_{j}(\xi,\cdot)\left\langle\tphi_{j}(\xi,\cdot);\ \cdot\ \right\rangle\\
&+&\displaystyle 
\chi(\xi)\ \sum_{j\in\{0,1,2\}}\ \eD^{\lambda_{j}(\xi)\,t}\frac{\phi_{j}(\xi,\cdot)-\beta_{1}^{(j)}(\xi)\,\uU_x}{\iD\uk\xi}\left\langle\tphi_{j}(\xi,\cdot);\ \cdot\ \right\rangle\\[0.5em]
(\tSe(t))_\xi&=&\displaystyle
\sum_{j\in\Z^*}\ \eD^{\lambdae{j}(\xi)\,t}\phie{j}(\xi,\cdot)\left\langle\tphie{j}(\xi,\cdot);\ \cdot\ \right\rangle\\
\end{array}
\ee
where $\chi$ is a suitable symmetric cut-off function and $\eD_1$ is the third vector of the canonical basis of $\C^3$. For proofs of Theorems~\ref{th:bounded} and~\ref{th:asymptotic} we could have replaced in above definitions $\beta^{(j)}(\xi)$ with $\beta^{(j)}(0)$ to deal with 
$$
\phi_{j}(\xi,\cdot)-\beta_{0}^{(j)}(0)\,\uU_x\,=\,
\phi_{j}(\xi,\cdot)-\phi_{j}(0,\cdot)\,.
$$
The above definitions prove to be more convenient only when we turn to analyze asymptotic behavior of solutions.

\subsection{Large-eigenvalue expansions}\label{s:high-freq}

In contrast with expansions when $(\lambda,\xi)\to(0,0)$, expansions when $|\lambda|\to\infty$ are readily obtained from explicit formulas of Subsection~\ref{s:spec}. Yet since they play a prominent role in the analysis we find convenient to state at least some of them explicitly here. Note also that, while we derive them here from explicit formulas, those expansions are in principle accessible by a direct analysis, not relying on integrability.

We start with a basic lemma.

\bl
Uniformly in $\xi\in(-\pi,\pi)$,
$$
\lambdae{j}(\xi)\ \stackrel{|j|\to\infty}{=}\ (\uk(2\pi\,j+\xi))^3\,+\,\cO(j)\,.
$$
\el

\begin{proof}
We focus on the limit $j\to\infty$, the full result being then derived from the symmetry $\lambdae{j}(\xi)=-\lambdae{-j}(-\xi)$. In this case the lemma follows from the fact that in terms of the Lax spectral parameter $\eta$ we have on one hand
\be\label{e:lambda-eta}
\frac{\lambda}{\iD|\eta|^{\tfrac32}}\ \stackrel{\eta\to-\infty}{=}\ 1\,+\,\cO\left(\frac{1}{|\eta|}\right)
\ee
and, on the other hand, 
$$
\iD(2\pi(j+1)+\xi)\ \stackrel{\eta\to-\infty}{=}\ -\frac{\lambda}{\uk\eta},+\,\cO\left(\frac{|\lambda|}{|\eta|^2}\right)
$$
so that
\be\label{e:j-eta}
\iD\uk(2\pi(j+1)+\xi)\ \stackrel{\eta\to-\infty}{=}\ |\eta|^{\tfrac12}\,+\,\cO\left(|\eta|^{-\tfrac12}\right)\,.
\ee
Note that here we have left implicit the dependence of $\lambda$ and $\eta$ on $(j,\xi)$. For the sake of concision in formulas, we shall do so repeatedly from now on.
\end{proof}

The foregoing lemma quantifies at the level of eigenvalues how close $L_\xi$ is from $-\uk^3(\d_x+\iD\xi)^3$. We now provides similar results at the level of eigenfunctions.

\bl\label{prop:ortho}
For any $s\in\N$, uniformly in $(\xi,x)\in(-\pi,\pi)\times\R$,
$$
\d_x^s\phie{j}(\xi,x)\ \stackrel{|j|\to\infty}{=}\ (2\iD\pi\,j)^s\,\phie{j}(\xi,x)\,+\,\cO(|j|^{s-1})\,.
$$
Uniformly in $(\xi,x)\in(-\pi,\pi)\times\R$,
$$
\tphie{j}(\xi,x)\ \stackrel{|j|\to\infty}{=}\ \phie{j}(\xi,x)\,+\,\cO(|j|^{-2})\,.
$$
\el

\begin{proof}
Again we may focus on the limit $j\to\infty$, relying this time on symmetries $\phie{j}(\xi,x)=\overline{\phie{-j}(-\xi,x)}$ and $\tphie{j}(\xi,x)=\overline{\tphie{-j}(-\xi,x)}$. The first expansion is then derived from the fact that $\phie{j}(\xi,x)$ equals
$$
\eD^{2\iD\pi\,x\,(j+1)}\ \times\ \left(1-\tfrac{\uk}{3}\lambda^{-1}\uU_x(x)\right)\ \times\ \exp\left(\frac{-\lambda}{\uk}\int_0^x\int_0^1\dfrac{\tfrac13(\uU(\theta)-\uU(\theta'))\ \dD \theta\ \dD \theta'}{(\eta-\uc+\tfrac13\uU(\theta))\,(\eta-\uc+\tfrac13\uU(\theta'))}\right)
$$
combined with estimates \eqref{e:lambda-eta}-\eqref{e:j-eta}. Along the same lines the second one follows readily from the fact that $\tphie{j}(\xi,x)$ equals
$$
\eD^{2\iD\pi\,x\,(j+1)}\ \times\ \left(1+
\dfrac{\uU(x)-\int_0^1\uU}{3(\eta-\uc)+\int_0^1\uU}\right)
\ \times\ \exp\left(\frac{-\lambda}{\uk}\int_0^x\int_0^1\dfrac{\tfrac13(\uU(\theta)-\uU(\theta'))\ \dD \theta\ \dD \theta'}{(\eta-\uc+\tfrac13\uU(\theta))\,(\eta-\uc+\tfrac13\uU(\theta'))}\right)
$$
combined with the foregoing expansions.
\end{proof}

\begin{proposition}\label{prop:tphi-expansion}
There exist a family of smooth functions of period one $(r_\ell)_{\ell\in\N}$ and a sequence of coefficients $(a_\ell(j,\xi))_{(\ell,j,\xi)\in\N\times\Z^*\times(-\pi,\pi)}$ such that for any $\ell\in\N$, uniformly in $\xi\in(-\pi,\pi)$
$$
a_\ell(j,\xi)\ \stackrel{|j|\to\infty}{=}\ \grandO\left(\frac{1}{|j|^\ell}\right)
$$
and for any $s\in\N$, there exists $R_s(j,\xi,x)$ such that
$$
\tphie{j}(\xi,x)\ =\ \sum_{\ell=0}^sa_\ell(j,\xi)\,\exp\left(2\iD\pi\,x\,j\right)\,r_\ell(x)\,+\,R_s(j,\xi,x)
$$
and uniformly in $(\xi,x)\in(-\pi,\pi)\times\R$
$$
R_s(j,\xi,x)\ \stackrel{|j|\to\infty}{=}\ \grandO\left(\frac{1}{|j|^{s+1}}\right)
$$
\end{proposition}

\begin{proof}
Again by using symmetries we may restrict to $j\in\N^*$. To turn expansions near infinity into expansions near zero we introduce $\Phi(\tau,x)$ defined as 
$$
\begin{array}{l}
\displaystyle
\eD^{2\iD\pi\,x}\ \times\ \left(1+\tau^2
\dfrac{-\uU(x)+\int_0^1\uU}{3+3\tau^2\uc-\tau^2\int_0^1\uU}\right)\\[1em]
\displaystyle
\times\ \exp\left(\tau\frac{\sqrt{(1+\tau^2\eta_1)\,(1+\tau^2\eta_2)\,(1+\tau^2\eta_3)}}{\uk}\int_0^x\int_0^1\dfrac{\tfrac13(\uU(\theta)-\uU(\theta'))\ \dD \theta\ \dD \theta'}{(1+\tau^2\uc-\tfrac13\tau^2\uU(\theta))\,(1+\tau^2\uc-\tfrac13\tau^2\uU(\theta'))}\right)
\end{array}
$$
so that for any $(j,\xi)\in\N^*\times(-\pi,\pi)$
$$
\tphie{j}(\xi,x)\ =\ \eD^{2\iD\pi\,x\,j}\ \Phi\left(|\etae{j}(\xi)|^{-\frac12},x\right)\,.
$$
Note that $\Phi$ is smooth in both variables on $(-|\eta_1|^{-\frac12},|\eta_1|^{-\frac12})\times \R$ and periodic of period one in its second variable. The result is now obtained with functions
$$
r_\ell(x)\ =\ \frac{1}{\ell!}\ \d_\tau^\ell \Phi(0,x)\,,\qquad (\ell,x)\in\N^*\times\R\,,
$$
coefficients
$$
a_\ell(j,\xi)\ =\ |\etae{j}(\xi)|^{-\frac\ell2}\,,\qquad (\ell,j,\xi)\in\N\times\N^*\times(-\pi,\pi)\,,
$$
and remainders
$$
R_s(j,\xi,x)\ =\ |\etae{j}(\xi)|^{-\frac{s+1}{2}}\ 
\int_0^1\frac{\d_\tau^{s+1}\Phi(t\,|\etae{j}(\xi)|^{-\tfrac{1}{2}},x)}{s!}\,(1-t)^s\,\dD t\,.
$$
\end{proof}

In Subsection \ref{s:side-band}, we have exhibited the very special role played by $\uU_x$ in the spectral decomposition of $L_0$. To make the most of the associated cancellations for non zero Floquet exponents we then need a uniform control on how left eigenfunctions vary with $\xi$. The following expansion provides such a control.

\begin{proposition}\label{prop:tphi-xi-expansion}
There exist a family of smooth functions of period one $(r_\ell)_{\ell\in\N^*}$ and a sequence of coefficients $(b_\ell(j,\xi))_{(\ell,j,\xi)\in\N^*\times\Z^*\times(-\pi,\pi)}$ such that for any $\ell\in\N^*$, uniformly in $\xi\in(-\pi,\pi)$
$$
b_\ell(j,\xi)\ \stackrel{|j|\to\infty}{=}\ \grandO\left(\frac{|\xi|}{|j|^{\ell+1}}\right)
$$
and for any $s\in\N$, there exists $\tR_s(j,\xi,x)$ such that
$$
\tphie{j}(\xi,x)\ =\ \tphie{j}(0,x)
\,+\,\sum_{\ell=1}^sb_\ell(j,\xi)\,\exp\left(2\iD\pi\,x\,j\right)\,r_\ell(x)\,+\,\tR_s(j,\xi,x)
$$
and uniformly in $(\xi,x)\in(-\pi,\pi)\times\R$
$$
\tR_s(j,\xi,x)\ \stackrel{|j|\to\infty}{=}\ \grandO\left(\frac{|\xi|}{|j|^{s+2}}\right)
$$
\end{proposition}

\begin{proof}
With $\Phi$ as in the foregoing proof the result is obtained with functions
$$
r_\ell(x)\ =\ \frac{1}{\ell!}\ \d_\tau^\ell \Phi(0,x)\,,\qquad (\ell,x)\in\N^*\times\R\,,
$$
coefficients
$$
b_\ell(j,\xi)\ =\ |\etae{j}(\xi)|^{-\frac\ell2}-|\etae{j}(0)|^{-\frac\ell2}\,,\qquad (\ell,j,\xi)\in\N^*\times\N^*\times(-\pi,\pi)\,,
$$
and remainders $\tR_s(\xi,x)$ given by
$$
\begin{array}{l}
\displaystyle
\left(|\etae{j}(\xi)|^{-\frac{s+1}{2}}-|\etae{j}(0)|^{-\frac{s+1}{2}}\right)\ 
\int_0^1\frac{\d_\tau^{s+1}\Phi(t\,|\etae{j}(\xi)|^{-\tfrac{1}{2}},x)}{s!}\,(1-t)^s\,\dD t\\[1em]
\displaystyle
+\ |\etae{j}(0)|^{-\frac{s+1}{2}}\,\left(|\etae{j}(\xi)|^{-\frac{1}{2}}-|\etae{j}(0)|^{-\frac{1}{2}}\right)\\[1em]
\displaystyle
\qquad\times\int_0^1\int_0^1\frac{\d_\tau^{s+2}\Phi(t\,(|\etae{j}(0)|^{-\tfrac{1}{2}}+\sigma\,(|\etae{j}(\xi)|^{-\tfrac{1}{2}}-|\etae{j}(0)|^{-\tfrac{1}{2}}),x)}{s!}\,t\,(1-t)^s\,\dD\sigma\,\dD t
\end{array}
$$
since, as follows from an elementary computation, $\d_\xi(|\etae{j}|^{\tfrac12})$ converges uniformly on $(-\pi,\pi)$ to $\uk^{-1}$ as $j\to\infty$. Indeed for some explicit smooth function $J$ satisfying $J(0)=1$ stands
$$
\uk\xi\ =\ |\etae{j}(\xi)|^{\frac{1}{2}}\ J(|\etae{j}(\xi)|^{-1})-|\etae{j}(0)|^{\frac{1}{2}}\ J(|\etae{j}(0)|^{-1})\,.
$$
\end{proof}

\subsection{Kernel representation}\label{s:Green}

We shall prove dispersive decay by estimating spatial representation of the solution. To do so we need to convert the foregoing purely spectral description in terms of Green kernels. 

The first easy key observation is that expanding Bloch transform definition \eqref{Bloch} yields for any smooth periodic function $\tphi$, any smooth localized $g$ and any $\xi\in(-\pi,\pi)$
$$
\langle \tphi;\check{g}(\xi,\cdot)\rangle
\ =\ \int_\R \eD^{-\ii\xi z}\ \overline{\tphi(z)}\,g(z)\ \dD\xi
$$
since Fourier series inversion formula provides for any $z$
$$
\frac{1}{2\pi}\sum_{\ell\in\Z}\int_0^1\eD^{-2\iD\pi\ell(y-z)}\,\tphi(y)\,\dD y
\,=\, \tphi(z)\,.
$$

With the purpose of analyzing oscillatory integrals where amplitudes are separated from oscillations we also pull out from Floquet right and left eigenfunctions some oscillating factors. This factorization allows for a consistent gluing of these Floquet eigenfunctions. Therefore for $(j,\xi)\in \Z\times [-\pi,\pi)\setminus\{(0,0)\}$ we set for any $x\in\R$
$$
\Lambdae(2\pi\,j+\xi)\,=\,\lambdae{j}(\xi)\,,\quad
\Phie(2\pi\,j+\xi,x)\,=\,\phie{j}(\xi,x)\,\eD^{-2\iD\pi\,j\,x}\,,\quad
\tPhie(2\pi\,j+\xi,x)\,=\,\tphie{j}(\xi,x)\,\eD^{-2\iD\pi\,j\,x}\,.
$$
Likewise for $\xi\in[-\pi,0)\cup(0,\pi)$ and $x\in\R$ we set
$$
\Lambdap(\xi)\,=\,\lambdap{1}(\xi)\,,\qquad
\Phip(\xi,x)\,=\,\phip{1}(\xi,x)\,,\qquad
\tPhip(\xi,x)\,=\,\tphip{1}(\xi,x)\,,
$$
and for $(j,\xi)\in\{1\}\times[-\pi,0)\cup\{-1\}\times(0,\pi)$
$$
\Lambdap(2\pi\,j+\xi)\,=\,\lambdap{2}(\xi)\,,\quad
\Phip(2\pi\,j+\xi,x)\,=\,\phip{2}(\xi,x)\,\eD^{-2\iD\pi\,j\,x}\,,\quad
\tPhip(2\pi\,j+\xi,x)\,=\,\tphip{2}(\xi,x)\,\eD^{-2\iD\pi\,j\,x}\,.
$$
Accordingly we build a "glued" version of the cut-off function introduced in Section~\ref{s:side-band} by setting
$$
\tchi\,=\,\chi\,+\,\chi(\cdot+2\pi)\,+\,\chi(\cdot-2\pi)\,.
$$

After this preparation we may define in a distributional\footnote{It turns out that this actually defines a bounded function $\Ke(t,\cdot,\cdot)$ when $t\neq0$.} sense %for $(t,x,y)\in\R^3$
$$
\Ke(t,x,y)\,=\,\int_\R 
\eD^{\Lambdae(\xi)\,t\,+\iD\,\xi\,(x-y)}\,(1-\chi(\xi))\,
\Phie(\xi,x)\,\overline{\tPhie(\xi,y)}\,\dD \xi
$$
and in a point-wise sense
$$
\Kp(t,x,y)\,=\,\int_{-2\pi}^{2\pi} 
\eD^{\Lambdap(\xi)\,t\,+\iD\,\xi\,(x-y)}\,(1-\tchi(\xi))\,
\Phip(\xi,x)\,\overline{\tPhip(\xi,y)}\,\dD \xi
$$
and for $j\in\{0,1,2\}$
$$
\Ko{j}(t,x,y)\,=\,\int_{-\pi}^{\pi} 
\eD^{\lambdap{j}(\xi)\,t\,+\iD\,\xi\,(x-y)}\,\chi(\xi)\ 
\frac{\phi_{j}(\xi,x)-\beta_{1}^{(j)}(\xi)\,\uU_x(x)}{\iD\uk\xi}\ \overline{\tphip{j}(\xi,y)}\,\dD \xi\,.
$$
This ensures that for any smooth and localized $g$ and any $(t,x)$
$$
\tS(t)(g)(x)\,=\,\int_\R\,\Big(\Ke(t,x,y)+\Kp(t,x,y)\,+\sum_{j\in\{0,1,2\}}\Ko{j}(t,x,y)\Big)\,g(y)\,\dD y\,.
$$

We need a similar description for the action on $\psi\uU_x$ in terms of $\psi_x$. It follows from the discussion in Section~\ref{s:slow-reduction} that we may focus on the case where $\psi$ is low-frequency and centered. To quantify the corresponding gain we choose a symmetric cut-off function $\chi_0$ such that $\textrm{supp }\chi\subset\textrm{supp}(1-\chi_0)$ and $\textrm{supp }\chi_0\subset (-\pi,\pi)$. Then if $\psi$ is centered, low-frequency in the sense that $\textrm{supp }\hat{\psi}\subset\textrm{supp }\chi$ and such that $\d_x\psi$ is localized, then for any $(t,x)$
$$
\tS(t)(\psi\,\uU_x)(x)\,=\,\int_\R\,\Big(\Ge(t,x,y)+\Gp(t,x,y)\,+\sum_{j\in\{0,1,2\}}\Go{j}(t,x,y)\Big)\,\psi_x(y)\,\dD y
$$
where $\Ge$, $\Gp$ and $\Go{}$ are defined in a point-wise way respectively by
\be
\Ge(t,x,y)\,=\,\sum_{j\in\Z}\int_{-\pi}^\pi 
\eD^{\lambdae{j}(\xi)\,t\,+\iD\,\xi\,(x-y)\,+2\iD\pi j\,x}\,\Ae{j}(\xi,x)\,\dD \xi
\ee
with amplitudes
$$
\Ae{j}(\xi,x)\,=\,\chi_0(\xi)\,(1-\chi(2\,\pi j+\xi))\,\Phie(2\pi\,j+\xi,x)\,
\left\langle \eD^{2\iD\pi j\,\,\cdot\,}\,\tfrac{\tPhie(2\pi\,j+\xi,\,\cdot\,)-\tPhie(2\pi\,j,\,\cdot\,)}{-\iD\xi} ;\uU_x \right\rangle
$$
by
\be
\Gp(t,x,y)\,=\,\sum_{j\in\{1,2\}}\int_{-\pi}^{\pi} 
\eD^{\lambdap{j}(\xi)\,t\,+\iD\,\xi\,(x-y)}\,\Ap{j}(\xi,x)\,\dD \xi
\ee
with amplitudes
$$
\Ap{j}(\xi,x)\,=\,
\,\chi_0(\xi)(1-\chi(\xi))\,
\phip{j}(\xi,x)\,\left\langle\tfrac{\tphip{j}(\xi,\,\cdot\,)-\tphip{j}(0,\,\cdot\,)}{-\iD\xi};\uU_x\right\rangle
$$
and for $j\in\{0,1,2\}$ by
\be
\Go{j}(t,x,y)\,=\,\int_{-\pi}^{\pi} 
\eD^{\lambdap{j}(\xi)\,t\,+\iD\,\xi\,(x-y)}\,\Ao{j}(\xi,x)\,\dD \xi
\ee
with amplitude
$$
\Ao{j}(\xi,x)\,=\,\chi(\xi)\,
\frac{\phi_{j}(\xi,x)-\beta_{1}^{(j)}(\xi)\,\uU_x(x)}{\iD\uk\xi}\,\left\langle\tfrac{\tphip{j}(\xi,\,\cdot\,)-\tphip{j}(0,\,\cdot\,)}{-\iD\xi};\uU_x\right\rangle\,.
$$

For easy reference we also state within our new set of notation corollaries of the analysis carried out along the proof of Propositions~\ref{prop:tphi-expansion} and~\ref{prop:tphi-xi-expansion} %foregoing sections or of slight variations of it 
and also needed in the study of foregoing kernels.

\begin{proposition}
%$\Phip$ and $\tPhip$ are smooth functions on $\supp (1-\tchi)\,\R^2$. 
$\Phie$, $\tPhie$ and $\d_\xi \tPhie$ are uniformly bounded on $\textrm{supp}(1-\chi)\times\R$ and, uniformly in $x\in\R$,
$$
|\d_\xi\Phie(\xi,x)|\,+\,|\d_\xi\tPhie(\xi,x)|\,+\,|\d_\xi^2\tPhie(\xi,x)|
\stackrel{|\xi|\to\infty}{=}\ \grandO\left(|\xi|^{-2}\right)\,.
$$
\end{proposition}

At last representations corresponding to $\spp$ may also be obtained. Explicitly, for $j\in\{0,1,2\}$, when $g$ is smooth and localized 
$$
\uk\d_x\spp_j(t)(g)(x)\,=\,
\eD_1\cdot\int_\R\ \ko{j}(t,x,y)\ g(y)\ \dD y
$$
with 
$$
\ko{j}(t,x,y)\,=\,\int_{-\pi}^\pi \eD^{\iD\xi(x-y)+\lambda_j(\xi)\,t}\,\chi(\xi)\ \beta^{(j)}(\xi)\ 
\overline{\tphi_j(\xi,y)}\ \dD \xi\,;
$$
and, when $\psi$ is centered, $\d_x\psi$ is localized and $\textrm{supp }\hat{\psi}\subset\textrm{supp }\chi$
$$
\uk\d_x\spp_j(t)(\psi\,\uU_x)\,=\,
\eD_1\cdot\int_\R\ \go{j}(t,x,y)\ \psi_x(y)\ \dD y
$$
with
$$
\go{j}(t,x,y)\,=\,\int_{-\pi}^{\pi} 
\eD^{\iD\,\xi\,(x-y)\,+\,\lambda_j(\xi)\,t}\,
\chi(\xi)\ \beta^{(j)}(\xi)\ 
\left\langle\tfrac{\tphi_j(\xi,\,\cdot\,)-\tphi_j(0,\,\cdot\,)}{-\iD\xi};\uU_x\right\rangle\,\dD \xi\,.
$$

%%%%%%%%%%%%%%%%%%%%%%%%%%%%%%%%%%%%%%%%%%%%%%%%%%%%%% 
%        STABILITY                                   %
%%%%%%%%%%%%%%%%%%%%%%%%%%%%%%%%%%%%%%%%%%%%%%%%%%%%%%

\section{Stability estimates}\label{s:stability}

Here begins the core of our proofs.

\subsection{Reduction to centered slow space modulation}\label{s:slow-reduction}

We first show that as far as stability results in $X=W^{s,p}(\R)$, $(s,p)\in\N\times[1,\infty]$, are concerned we may replace the definition \eqref{sm_norm} with 
\be\label{sm_norm_lf}
N_X(W)\ =\ \inf_{\substack{W=V+\uU_x\psi\\\psi(\infty)=-\psi(-\infty)
\\\psi \textrm{ is low frequency}}} 
\|V\|_{X}\ +\ \|\psi_x\|_{X}\,.
\ee

By $\psi(\infty)=-\psi(-\infty)$ we mean that $\psi=\d_x^{-1}\psi_x$, where $\d_x^{-1}$ is defined as a principal value on the Fourier side. Obviously this constraint does not restrict potential applications of linear estimates to a nonlinear analysis since at the nonlinear level it could be achieved initially by translating $\uU$ by $\tfrac12(\psi_0(\infty)+\psi_0(-\infty))$. However it may also be achieved directly at the linear level since $S(t)$ commutes with translation by a constant multiple of $\uU_x$.

As for being low-frequency, it means here having a Fourier transform supported in $(-\pi,\pi)$. We shall exhibit decompositions that satisfy this extra constraint so the only thing to be proved here is that one may replace any decomposition $V_0+\uU_x\psi_0$ with a decomposition $\tV_0+\uU_x\tpsi_0$ satisfying the extra constraint and such that
$$
\|\tV_0\|_{X}\ +\ \|(\tpsi_0)_x\|_{X}\ \leq\ C\,(\|V_0\|_{X}\ +\ \|(\psi_0)_x\|_{X})
$$
when $X$ is any $W^{s,p}(\R)$ and with a constant $C$ depending only on $X$. To do so we introduce localizations in frequencies $(\psi_0)^{LF}$ and $(\psi_0)^{HF}$ defined on the Fourier side by, for any $\xi\in\R$,
$$
\widehat{(\psi_0)^{LF}}(\xi)\ =\ \chi(\xi)\,\widehat{\psi_0}(\xi)
\qquad\textrm{and}\qquad
\widehat{(\psi_0)^{HF}}(\xi)\ =\ (1-\chi(\xi))\,\widehat{\psi_0}(\xi)
$$
where $\chi$ is a suitable symmetric cut-off function. Then we set
$$
\tV_0\ =\ V_0\,+\,(\psi_0)^{HF}\uU_x
\qquad\textrm{and}\qquad
\tpsi_0\ =\ (\psi_0)^{LF}\,.
$$
This achieves a suitable decomposition since both $\d_x(\psi_0)^{LF}$ and $(\psi_0)^{HF}$ are obtained from $\d_x\psi_0$ by a convolution with a $L^1$ function as both $\xi\mapsto \chi(\xi)$ and $\xi\mapsto (1-\chi(\xi))\,\xi^{-1}$ lie in $H^1(\R)$ and $x\mapsto (1+|x|)^{-1}$ is square-integrable.

\subsection{Bounded stability}\label{s:bounded-stability}

Since the evolution is unitary in coordinates 
$$(\langle\tphie{j}(\xi,\cdot);\check{W}(\xi,\cdot)\rangle)_{(j,\xi)\in\Z^*\times(-\pi,\pi)}\,,\qquad(\langle\tphi_{j}(\xi,\cdot);\check{W}(\xi,\cdot)\rangle)_{(j,\xi)\in\{0,1,2\}\times(-\pi,\pi)}\,,$$ 
Theorem~\ref{th:bounded} can be proved by showing that $N_{H^s(\R)}$ may be equivalently written in terms of those coordinates.

To do so we define, for $s\in\R_+$, $\|\,\cdot\,\|_{X^s}$ through 
$$
\begin{array}{rcl}\displaystyle
\|W\|_{X^s}^2&=&\displaystyle
\left\|(j,\xi)\ \mapsto\ (2\pi\,j)^{s}\left\langle\tphie{j}(\xi,\cdot);\check{W}(\xi,\cdot)\right\rangle\right\|_{\ell^2(\Z^*;L^2([-\pi,\pi]))}^2
\\[1em]
&&\displaystyle
\quad+\quad\left\|(j,\xi)\ \mapsto\ \left\langle\tphi_j(\xi,\cdot);\check{W}(\xi,\cdot)\right\rangle\right\|_{\ell^2(\{0,1,2\};L^2([-\pi,\pi]))}^2\,.
\end{array}
$$
Theorem~\ref{th:bounded} follows from the following proposition.

\begin{proposition}\label{prop:bounded}
For any $s\in\N$ there exist positive $C$ and $C'$ such that 
$$
C\,\|\cdot\|_{X^s}\ \leq\ N_{H^s(\R)}(\cdot)\ \leq\ C'\,\|\cdot\|_{X^s}\,.
$$
\end{proposition}

\begin{proof}
As follows from the discussion in Subsection~\ref{s:slow-reduction}, we may safely use definition in \eqref{sm_norm_lf}. We first prove the right-hand side inequality and starts by the case $s=0$. Since 
$$
\Id=S(0)=\uU_x\eD_1\cdot\spp(0)+\tS(0)
$$ 
and $\spp(0)$ provides centered low-frequency phases we only need to prove that there exists some constant $C$ such that 
$$
\begin{array}{rcl}
\|\tS(0)(W_0)\|_{L^2(\R)}&\leq&C\,\|W_0\|_{X^0}\\[1em]
\|\d_x\spp(0)(W_0)\|_{L^2(\R)}&\leq&C\,\|W_0\|_{X^0}\,.
\end{array}
$$
The latter inequality stems directly from classical Parseval identity with
$$
C\,=\,\sqrt{2\pi}\,\sqrt{3}\,\uk^{-1}\max_{(j,\xi)\{0,1,2\}\times[-\xi_0,\xi_0]}|\beta^{(j)}(\xi)|\,.
$$
Likewise  $\|\tSo(0)(W_0)\|_{L^2(\R)}$ is bounded by $C\,\|W_0\|_{X^0}$ with
$$
C\,=\,\sqrt{2\pi}\,\sqrt{3}\,\uk^{-1}\left[\max_{\substack{j\in\{0,1,2\}\\\xi\in[-\pi,\pi]}}|\chi'(\xi)|^2\|\phi_j(\xi,\cdot)\|_{L^2((0,1))}^2+\max_{\substack{j\in\{0,1,2\}\\\xi\in[-\pi,\pi]}}|\xi|^{-2}\|\phi_j(\xi,\cdot)-\beta_1^{(j)}(\xi)\uU_x\|_{L^2((0,1))}^2\right]^{1/2}
$$
by using Parseval identity for the Bloch transform. There only remains to bound $\|\tSe(0)(W_0)\|_{L^2(\R)}$. Applying again Parseval identity the result would follow from bounding $\|(\tSe(0)(W_0))\check{\ }\,(\xi,\cdot)\|_{L^2((0,1))}$ by a multiple uniform in $\xi$ of $\|j\,\mapsto\,\langle\tphie{j}(\xi,\cdot);\check{W_0}(\xi,\cdot)\rangle\|_{\ell^2(\Z^*)}$. Now for any $\xi\in[-\pi,\pi]$ bi-orthogonality relations imply
$$
\begin{array}{l}\ds
\|(\tSe(0)(W_0))\check{\ }\,\ds(\xi,\cdot)\|_{L^2((0,1))}^2\\[1em]\ds
\,=\,
\sum_{j\in\Z^*}\left|\left\langle\tphie{j}(\xi,\cdot);\check{W_0}(\xi,\cdot)\right\rangle\right|^2\\
\ds\hspace{4em}
+\,\left\langle\sum_{j\in\Z^*}(\phie{j}(\xi,\cdot)-\tphie{j}(\xi,\cdot))\left\langle\tphie{j}(\xi,\cdot);\check{W_0}(\xi,\cdot)\right\rangle;
(\tSe(0)(W_0))\check{\ }\,(\xi,\cdot)\right\rangle\\[1em]
\ds\leq
\sum_{j\in\Z^*}|\langle\tphie{j}(\xi,\cdot);\check{W_0}(\xi,\cdot)\rangle|^2
+\,C_0
\big(\sum_{j\in\Z^*}|\langle\tphie{j}(\xi,\cdot);\check{W_0}(\xi,\cdot)\rangle|^2\big)^{1/2}
\|(\tSe(0)(W_0))\check{\ }\,(\xi,\cdot)\|_{L^2((0,1))}
\end{array}
$$
with
$$
C_0\ =\ \sup_{\zeta\in[-\pi,\pi]} \big(\sum_{j\in\Z^*}\|\phie{j}(\zeta,\cdot)-\tphie{j}(\zeta,\cdot)\|_{L^2((0,1))}^2\big)^{1/2}
$$
that is indeed finite by Proposition~\ref{prop:ortho}. This concludes the proof of the right-hand inequality when $s=0$.

We now explain how to extend the foregoing analysis to general $s$. A slight variation on above arguments show that actually both $\|\d_x\spp(0)(W_0)\|_{H^s(\R)}$ and $\|\tSo(0)(W_0)\|_{H^s(\R)}$ are bounded by a multiple of $\|W_0\|_{X^0}$ so that we may focus on bounding  $\|\tSe(0)(W_0)\|_{H^s(\R)}$. To rely on Parseval identity we aim at bounding $\|(\d_x+\iD\xi)^s(\tSe(0)(W_0))\check{\ }\,(\xi,\cdot)\|_{L^2((0,1))}$. To do so we expand
$$
\begin{array}{l}\ds
\|(\d_x+\iD\xi)^s(\tSe(0)(W_0))\check{\ }\,(\xi,\cdot)\|_{L^2((0,1))}\\[1em]\ds
\,\leq\,
\left\|\sum_{j\in\Z^*}(2\pi\iD j+\iD\xi)^s\phie{j}(\xi,\cdot)\left\langle\tphie{j}(\xi,\cdot);\check{W_0}(\xi,\cdot)\right\rangle\right\|_{L^2((0,1))}
\\
\ds
+\,\left\|\sum_{j\in\Z^*}((\d_x+\iD\xi)^s\phie{j}(\xi,\cdot)-(2\pi\iD j+\iD\xi)^s\phie{j}(\xi,\cdot))\left\langle\tphie{j}(\xi,\cdot);\check{W_0}(\xi,\cdot)\right\rangle\right\|_{L^2((0,1))}\,.
\end{array}
$$
It follows from the $s=0$ analysis that the first term on the right-hand side inequality is bounded by a multiple of $\|j\,\mapsto\,(2\pi\iD j+\iD\xi)^s\langle\tphie{j}(\xi,\cdot);\check{W_0}(\xi,\cdot)\rangle\|_{\ell^2(\Z)}$. Moreover Proposition~\ref{prop:ortho} yields that the second one is also bounded by a multiple of $\|j\,\mapsto\,|j|^s\langle\tphie{j}(\xi,\cdot);\check{W_0}(\xi,\cdot)\rangle\|_{\ell^2(\Z)}$. This is sufficient to conclude the proof of the second inequality in Proposition~\ref{prop:bounded}.

We now turn to the proof of the first inequality. We must prove that there exists some constant $C$ such that for any $V_0\in H^s(\R)$ and any $\psi_0$ low-frequency, centered and such that $\d_x\psi_0\in H^s(\R)$
$$
\|V_0\|_{X^s}\ \leq\ C\,\|V_0\|_{H^s(\R)}\,,\qquad\qquad
\|\uU_x\psi_0\|_{X^s}\ \leq\ C\,\|\d_x\psi_0\|_{H^s(\R)}\,.
$$
We derive the former inequality essentially from Proposition~\ref{prop:tphi-expansion}. First observe that uniformly in $\xi\in(-\pi,\pi)$
$$
\left\|j\ \mapsto\ \left\langle\tphi_j(\xi,\cdot);\check{V_0}(\xi,\cdot)\right\rangle\right\|_{\ell^2(\{0,1,2\})}
\leq C\,\|\check{V_0}(\xi,\cdot)\|_{L^2((0,1))}
$$
with
$$C\ =\ 
\max_{\eta\in(-\pi,\pi)}\left\|(j,x)\ \mapsto\ \tphi_j(\eta,x)\right\|_{\ell^2(\{0,1,2\};L^2((0,1)))}
$$
Moreover with notation of Proposition~\ref{prop:tphi-expansion} uniformly in $\xi\in(-\pi,\pi)$
$$
\left\|j\ \mapsto\ (2\pi\,j)^{s}\left\langle\tphie{j}(\xi,\cdot);\check{V_0}(\xi,\cdot)\right\rangle\right\|_{\ell^2(\Z^*)}
\leq \sum_{\ell=0}^s C_{\ell}\,\|\overline{r_\ell}\,\check{V_0}(\xi,\cdot)\|_{H^{s-\ell}((0,1))}\,+\,C'_s\,\|\check{V_0}(\xi,\cdot)\|_{L^{2}((0,1))}
$$
with 
$$
C_{\ell}\,=\,\sup_{\substack{j\in\Z^*\\\zeta\in(-\pi,\pi)}}(2\pi|j|)^{\ell}\,|a_\ell(j,\zeta)|\,,\qquad 0\leq\ell\leq s\,,
$$
and
$$
C'_s\,=\,\big\|\,j\ \mapsto\ (2\pi\,j)^{s}\sup_{\zeta\in(-\pi,\pi)}\|R_s(j,\zeta,\cdot)\|_{L^{2}((0,1))}\big\|_{\ell^2(\Z^*)}\,.
$$
Hence for some constant $C$, uniformly in $\xi\in(-\pi,\pi)$
$$
\left\|j\ \mapsto\ (2\pi\,j)^{s}\left\langle\tphie{j}(\xi,\cdot);\check{V_0}(\xi,\cdot)\right\rangle\right\|_{\ell^2(\Z^*)}
\leq C \left(\sum_{\ell=0}^s\|(\d_x+\iD\xi)^\ell\check{V_0}(\xi,\cdot)\|^2\right)^{1/2}\,.
$$
This enables us to achieve the proof of the claimed estimate by appealing to Parseval identity. To prove the remaining estimate we first observe that since $\psi_0$ is low-frequency we may benefit from orthogonality relations provided by Proposition~\ref{prop:slow-spec-expansion} to derive
$$
\left\langle\tphi(\xi,\cdot);(\uU_x\psi_0)\check{\ }\,(\xi,\cdot)\right\rangle\ =\ 
\widehat{(\d_x\psi_0)}(\xi)\ \left\langle\frac{\tphi(\xi,\cdot)-\tphi(0,\cdot)}{-\iD\xi};\uU_x\right\rangle
$$
both for $\tphi=\tphie{j}$, $j\in\Z^*$, and for $\tphi=\tphi_j$, $j\in\{0,1,2\}$. Then with notation of Proposition~\ref{prop:tphi-xi-expansion}, uniformly in $\xi\in(-\pi,\pi)$
$$
\begin{array}{l}\ds
\|j\,\mapsto\,(2\pi\,j)^{s}\langle\tphie{j}(\xi,\cdot);(\uU_x\psi_0)\check{\ }\,(\xi,\cdot)\rangle\|_{\ell^2(\Z^*)}\\[0.5em]\ds
\quad+\ \|j\,\mapsto\,\langle\tphi_j(\xi,\cdot);(\uU_x\psi_0)\check{\ }\,(\xi,\cdot)\rangle\|_{\ell^2(\{0,1,2\})}\quad\leq\quad C\,|\widehat{(\d_x\psi_0)}(\xi)|
\end{array}
$$
with
$$
\begin{array}{rcl}
C&=&\ds
\|j\mapsto (2\pi j)^{-1}\|_{\ell^2(\Z^*)}\sum_{\ell=1}^s\,\|\bar{r_\ell}\,\uU_x\|_{H^{s-\ell}((0,1))}\,\sup_{\substack{j\in\Z^*\\\zeta\in(-\pi,\pi)}}\tfrac{|j|^{\ell+1}}{|\zeta|}\,|b_\ell(j,\zeta)|\\[0.5em]
&+&\ds
\|j\mapsto (2\pi j)^{-2}\|_{\ell^1(\Z^*)}\,\,\|\uU_x\|_{L^2((0,1))}\,\sup_{\substack{j\in\Z^*\\\zeta\in(-\pi,\pi)}}\tfrac{(2\pi\,|j|)^{s+2}}{|\zeta|}\,\|\tR_s(j,\zeta,\cdot)\|_{L^2((0,1))}\\[0.5em]
&+&\ds
\|\uU_x\|_{L^2((0,1))}\,\|j\mapsto\max_{\zeta\in(-\pi,\pi)}\|\d_\xi \tphi_j(\zeta,\cdot)\|_{L^2((0,1))}\|_{{0,1,2}}\,.
\end{array}
$$ 
Then applying the classical Parseval identity achieves the proof of the proposition.
\end{proof}

\subsection{Asymptotic stability}\label{s:asymptotic-stability}

Theorem~\ref{th:asymptotic} follows readily from the following proposition.

\begin{proposition}\label{prop:asymptotic}
For any cnoidal wave satisfying condition~\ref{A}, there exists a constant $C$ such that for any $(t,x,y)\in\R^3$
$$
|t|^{1/2}\,|\Ke(t,x,y)|\,+\,(1+|t|)^{1/2}\,|\Kp(t,x,y)|\,+\,
(1+|t|)^{1/3}\,\sum_{j\in\{0,1,2\}}\,|\Ko{j}(t,x,y)|\,\leq C
$$
$$
|t|^{1/2}\,|\Ge(t,x,y)|\,+\,(1+|t|)^{1/2}\,|\Gp(t,x,y)|\,+\,
(1+|t|)^{1/3}\,\sum_{j\in\{0,1,2\}}\,|\Go{j}(t,x,y)|\,\leq C
$$
and
$$
(1+|t|)^{1/3}\,\sum_{j\in\{0,1,2\}}\,|\ko{j}(t,x,y)|\,+\,
(1+|t|)^{1/3}\,\sum_{j\in\{0,1,2\}}\,|\go{j}(t,x,y)|\,\leq C\,.
$$
\end{proposition}

In turn Proposition~\ref{prop:asymptotic} follows readily from estimates obtained in Section~\ref{s:preparation} and the classical van der Corput Lemma which we recall below ; see for instance \cite[Corollary~1.1]{Linares-Ponce} or Appendix~\ref{s:oscillatory} for a proof.

\bl\label{l:vdC}
Let $p\in\N\setminus\{0\}$. There exists $C$ such that for any closed interval $I$ and any smooth functions $a:\,I\to\R$ and $F:I\to\C$ such that
\begin{itemize}
\item $|a^{(p)}|$ is bounded away from $0$ on $I$ and, when $p=1$, $|a^{(p)}|$ is monotone and coercive
\item $F'$ is integrable 
\end{itemize}
then $\int_I \eD^{\ii\,a(\xi)}\ F(\xi)\ \dd\xi$ is well-defined and
$$
\left|\int_I \eD^{\ii\,a(\xi)}\ F(\xi)\ \dd\xi\right|
\ \leq\ \frac{C}{(\inf_I\,|a^{(p)}|)^{1/p}}\ \left[\,\sup_I|F|\,+\,\int_I|F'|\,\right]\,.
$$
\el

\section{Asymptotic behavior}\label{s:behavior}

Estimates of the foregoing section. 

\subsection{Spectral validity}\label{s:mod-spectral}

The following proposition is the key spectral observation leading to Theorem~\ref{th:behavior}. 

\begin{proposition}\label{prop:spec-modulation}
In Proposition~\ref{prop:slow-spec-expansion} one may choose
$$
(q_0(0,\cdot),q_{1}(0,\cdot),q_{2}(0,\cdot))
\,=\,(\Up{\uk,\uM,\uP}_x,\d_M\Up{\uk,\uM,\uP},\d_P\Up{\uk,\uM,\uP})
$$ 
and
$$
\d_\xi q_{0}(0,\cdot)\,=\,\iD\uk\,\d_k\Up{\uk,\uM,\uP}\,.
$$
Moreover, then, $(\beta_0(0),\beta_1(0),\beta_2(0))$ and $(\tbeta_0(0),\tbeta_1(0),\tbeta_2(0))$ are dual right and left eigenbases of
$$
-\uom\I-\uk\begin{pmatrix}\dd\omega(\ua)\\\dD P(\ua)\\\dd F(\ua)\end{pmatrix}
$$
associated with eigenvalues $(a_0^{(0)},a_0^{(1)},a_0^{(2)})$, that are such that
$$
\lambda_j(\xi)\,\stackrel{\xi\to0}{=}\,\iD\xi a_0^{(j)}+\grandO(|\xi|^3)\,.
$$
\end{proposition}

Again this is a corollary of the proof of \cite[Theorem~1]{Benzoni-Noble-Rodrigues}. 

\smallskip

The corresponding choice will be made from now on.

%TODO: 
%
%1. Say two words about the proof
%
%2. Provide explicit expression for D

\subsection{Slow modulation behavior}\label{s:mod-behavior}

To prove Theorem~\ref{th:behavior} we first choose $(\psi,M,P)$ according to
$$
\psi(t,\cdot)\,=\,\eD_1\cdot \spp(t)(W_0)
$$
and
$$
\bp\uk\d_x\psi(t,\cdot)\\M(t,\cdot)\\P(t,\cdot)\ep
\,=\,\uk\d_x\spp(t)(W_0)\,.
$$
In particular if we pick $(\psi_0,V_0)$ such that $W_0=\psi_0\,\uU_x+V_0$ and $\|\d_x\psi_0\|_{L^1(\R)}+\|V_0\|_{L^1(\R)}<+\infty$ then
$$
\bp\uk\d_x\psi(t,x)\\M(t,x)\\P(t,x)\ep
\,=\,\sum_{j\in\{0,1,2\}}
\int_\R [\ko{j}(t,x,y)\,V_0(y)+\go{j}(t,x,y)\,\d_x\psi_0(y)]\,\dD y
$$
so that estimates of Theorem~\ref{th:behavior} on $(\uk\d_x\psi,M,P)$ are corollaries of Proposition~\ref{prop:asymptotic}.

Actually Proposition~\ref{prop:asymptotic} contains already a significant part of Theorem~\ref{th:behavior}. The remaining part is to prove that the foregoing definitions do capture the main contributions corresponding to $\Ko{j}$ and $\Go{j}$ terms. For $j\in\{0,1,2\}$ we may explicitly derive from Proposition~\ref{prop:spec-modulation} that
$$
\begin{array}{rcl}\ds
\Ko{j}(t,x,y)&=&\ds
\bdD\Up{\uk,\uM,\uP}(x)\cdot(\ko{j}(t,x,y))\,+\,\tKo{j}(t,x,y)\\[0.5em]\ds
\Go{j}(t,x,y)&=&\ds
\bdD\Up{\uk,\uM,\uP}(x)\cdot(\go{j}(t,x,y))\,+\,\tGo{j}(t,x,y)
\end{array}
$$
with 
$$
\tKo{j}(t,x,y)\,=\,\int_{-\pi}^{\pi} 
\eD^{\lambdap{j}(\xi)\,t\,+\iD\,\xi\,(x-y)}\,\chi(\xi)\ 
\frac{\phi_{j,quad}(\xi,x)}{\iD\uk\xi}\ \overline{\tphip{j}(\xi,y)}\,\dD \xi\,.
$$
and
$$
\tGo{j}(t,x,y)\,=\,\int_{-\pi}^{\pi} 
\eD^{\lambdap{j}(\xi)\,t\,+\iD\,\xi\,(x-y)}\,\chi(\xi)\ 
\frac{\phi_{j,quad}(\xi,x)}{\iD\uk\xi}\ \left\langle\tfrac{\tphip{j}(\xi,\,\cdot\,)-\tphip{j}(0,\,\cdot\,)}{-\iD\xi};\uU_x\right\rangle\,\dD \xi\,
$$
where with notational conventions introduced in Proposition~\ref{prop:slow-spec-expansion}
$$
\phi_{j,quad}(\xi,\cdot)\,=\,
\beta_{0}^{(j)}(\xi)\quad [q_{0}(\xi,\cdot)-q_{0}(0,\cdot)
-\d_\xi q_{0}(0,\cdot)\,\xi]
\,+\,
(\iD\uk\xi)\sum_{l=1}^2\beta_l^{(j)}(\xi)\ [q_l(\xi,\cdot)-q_l(0,\cdot)]
\,.
$$

Now we achieve the proof of Theorem~\ref{th:behavior} with the following proposition.

\begin{proposition}\label{prop:refined-asymptotic}
For any cnoidal wave satisfying condition~\ref{A}, there exists a constant $C$ such that for any $(t,x,y)\in\R^3$
$$
(1+|t|)^{1/2}\,\sum_{j\in\{0,1,2\}}\,(\,|\tKo{j}(t,x,y)|+|\tGo{j}(t,x,y)|\,)\,\leq C\,.
$$
\end{proposition}

In turn the foregoing proposition follows from the following refined van der Corput lemma, applied with $p=3$, $\alpha=1$, $\xi_*=0$. See Appendix~\ref{s:oscillatory} for a proof of the lemma.

\begin{lemma}\label{l:refined-vdC}
Let $p\in\N\setminus\{0,1\}$ and $\alpha\in[0,1]$. There exists $C$ such that for any closed interval $I$ and any smooth functions $a:\,I\to\R$ and $F:I\to\C$ such that
\begin{itemize}
\item $|a^{(p)}|$ is bounded away from $0$ on $I$ 
\item $|a^{(p-1)}|$ vanishes at $\xi_*\in I$ such that $G:=F\times|\cdot-\xi_*|^{-\alpha}$ is bounded 
\item $F'$ is integrable
\end{itemize}
then
$$
\left|\int_I \eD^{\ii\,a(\xi)}\ F(\xi)\ \dd\xi\right|
\ \leq\ \frac{C}{(\inf_I\,|a^{(p)}|)^{(\alpha+1)/(p+\alpha(p-2))}}\ \left[\,\sup_I|F|\,+\,\int_I|F'|\,+\,\sup_I|G|\,\right]\,.
$$
\end{lemma}

\subsection{Modulation equations}\label{s:averaged}

The foregoing oscillatory lemma are well-adapted to expansions of eigenvectors. It already tell us that to prove Theorems~\ref{th:3rd} and~\ref{th:qth} one may replace $\sum_{j\in\{0,1,2\}}\ko{j}(t,x,y)$ with
$$
\begin{array}{rl}\ds
\sum_{j\in\{0,1,2\}}&\ds
\int_{-\pi}^\pi \eD^{\iD\xi(x-y)+\lambda_j(\xi)\,t}\,\chi(\xi)\ \beta^{(j)}(0)\ 
\overline{\tphi_j(0,y)}\ \dD \xi\\[0.5em]
=&\ds
\sum_{j\in\{0,1,2\}}\int_{-\pi}^\pi \eD^{\iD\xi(x-y)+\lambda_j(\xi)\,t}\,\chi(\xi)\ \beta^{(j)}(0)\,\tbeta^{(j)}(0)\cdot\begin{pmatrix}0\\1\\\uU(y)\end{pmatrix}\ \dD \xi
\end{array}
$$
and $\sum_{j\in\{0,1,2\}}\go{j}(t,x,y)$ with
$$
\begin{array}{rl}\ds
\sum_{j\in\{0,1,2\}}&\ds
\int_{-\pi}^\pi \eD^{\iD\xi(x-y)+\lambda_j(\xi)\,t}\,\chi(\xi)\ \beta^{(j)}(0)\ 
\left\langle\tfrac{\d_\xi\tphi_j(0,\,\cdot\,)}{-\iD};\uU_x\right\rangle\ \dD \xi\\[0.5em]
=&\ds
\sum_{j\in\{0,1,2\}}\int_{-\pi}^\pi \eD^{\iD\xi(x-y)+\lambda_j(\xi)\,t}\,\chi(\xi)\ \beta^{(j)}(0)\,\tbeta^{(j)}(0)\cdot\begin{pmatrix}\uk\\0\\0\end{pmatrix}\ \dD \xi
\end{array}
$$
since
$$
\begin{array}{rclclcl}\ds
\langle\,\d_\xi\tq_1(0);\uU_x\rangle
&=&\ds-\langle\,\tq_1(0);\d_\xi q_0(0)\rangle
&=&-\iD\uk\d_kM(\ua)&=&0\\[0.5em]\ds
\langle\,\d_\xi\tq_2(0);\uU_x\rangle
&=&-\langle\,\tq_2(0);\d_\xi q_0(0)\rangle
&=&-\iD\uk\d_kP(\ua)&=&0\,.
\end{array}
$$

For comparison, note that 
$$
\SigW_q(t)\chi(\iD^{-1}\d_x)(a)\,=\,\int_{\R}\,\sigma^{(q)}(t,x,y)\,a(y)\,\dD y
$$
with $\sigma^{(q)}\,=\,\sum_{j\in\{0,1,2\}}\sigma^{(q)}_j$ and, for $j\in\{0,1,2\}$,
$$
\int_{-\pi}^\pi \eD^{\iD\xi(x-y)+\lambda_j^{(q)}(\xi)\,t}\,\chi(\xi)\ \beta^{(j)}(0)\,\overline{\tbeta^{(j)}(0)}\,\dD \xi\,.
$$
Therefore the missing piece is a lemma allowing to measure the effect of expansions of eigenvalues. This is the purpose of the following lemma. Its proof is also given in Appendix~\ref{s:oscillatory}.

\begin{lemma}\label{l:phase-vdC}
Let $p\in\N\setminus\{0,1\}$, $q>p$, $\kappa>0$ and $M\in\R_+$. There exist positive $\eps_0$ and $C$ such that for any closed interval $I$ and any smooth functions $\omega:\,I\to\R$, $\omega_0:\,I\to\R$ and $F:I\to\C$ such that
\begin{itemize}
\item $|\omega^{(p)}|$ and $|\omega_0^{(p)}|$ are larger than $\kappa$ 
\item $|\omega^{(p-1)}|$ and $|\omega_0^{(p-1)}|$ vanish at $\xi_*\in I$ such that $I\subset [\xi_*-\eps_0,\xi_*+\eps_0]$
\item $F'$ is integrable on $I$
\item $|\omega^{(\ell)}-\omega_0^{(\ell)}|\times|\cdot-\xi_*|^{-q+\ell}$ is bounded by $M$ for any $0\leq\ell\leq p-2$
\end{itemize}
then for any $|t|\geq1$
$$
\left|\int_I (\eD^{\ii\,\omega(\xi)t}- \eD^{\ii\,\omega_0(\xi)t})\ F(\xi)\ \dd\xi\right|
\ \leq\ \frac{C}{|t|^{(q-1)/(q(p-1))}}\ \left[\,\sup_I|F|\,+\,\int_I|F'|\,\right]\,.
$$
\end{lemma}

\begin{lemma}\label{l:singular-phase-vdC}
Let $q$ be an odd integer larger than $3$, $\kappa>0$ and $M\in\R_+$. There exist positive $\eps_0$ and $C$ such that for any closed interval $I$ and any smooth functions $\omega:\,I\to\R$, $\omega_0:\,I\to\R$ and $F:I\to\C$ such that
\begin{itemize}
\item $|\omega^{'''}|$ and $|\omega_0^{'''}|$ are larger than $\kappa$ and smaller than $M$
\item $|\omega^{''}|$ and $|\omega_0^{''}|$ vanish at $\xi_*\in I$ such that $I\subset [\xi_*-\eps_0,\xi_*+\eps_0]$
\item $G:=F\times|\cdot-\xi_*|$ and $H:=F'\times|\cdot-\xi_*|^{2}$ are bounded
\item $|\omega^{(\ell)}-\omega_0^{(\ell)}|\times|\cdot-\xi_*|^{-q+\ell}$ is bounded by $M$ for any $0\leq\ell\leq 2$
\end{itemize}
then for any $|t|\geq1$
$$
\left|\int_I (\eD^{\ii\,\omega(\xi)t}- \eD^{\ii\,\omega_0(\xi)t})\ F(\xi)\ \dd\xi\right|
\ \leq\ 
\ds\frac{C}{|t|^{(q-3)/(2q)}}\ \left[\,\sup_I|G|\,+\,\sup_I|H|\,\right]\,.
$$
\end{lemma}

Those are directly applied to obtain Theorem~\ref{th:qth}. To conclude the proof of Theorem~\ref{th:3rd}, we only need to add that by using classical van der Corput lemma, one proves that one may actually replace $\SigW_3(t)\chi(\iD^{-1}\d_x)$ with $\SigW_3(t)$ at the level of comparison we aim at.

To prove Theorem~\ref{th:qth-phase}, the main change is that to apply Lemma~\ref{l:refined-vdC} and reach the required level of approximation one may only replace $\sum_{j\in\{0,1,2\}}\ko{j}(t,x,y)$ with
$$
\sum_{j\in\{0,1,2\}}
\int_{-\pi}^\pi \eD^{\iD\xi(x-y)+\lambda_j(\xi)\,t}\,\chi(\xi)\ \beta^{(j)}(0)\ 
\left(\overline{\tphi_j(0,y)}\,+\,\xi\,\overline{\d_\xi\tphi_j(0,y)}\right)\ \dD \xi
$$
and $\sum_{j\in\{0,1,2\}}\go{j}(t,x,y)$ with
$$
\sum_{j\in\{0,1,2\}}
\int_{-\pi}^\pi \eD^{\iD\xi(x-y)+\lambda_j(\xi)\,t}\,\chi(\xi)\ \beta^{(j)}(0)\ 
\left\langle\iD\left(\d_\xi\tphi_j(0,\,\cdot\,)+\xi\,\d_\xi^2\tphi_j(0,\,\cdot\,\right);\uU_x\right\rangle\ \dD \xi
$$
so that the theorem is proved with\footnote{Note that by using the invariance of the dual spectral problem under $(\lambda,\xi,\tphi)\mapsto (\overline{\lambda},-\xi,\overline{\tphi})$ one shows that $\d_\xi^{\ell}\tphi_j(0,\cdot)$ is real when $\ell$ is even and purely imaginary when $\ell$ is odd.}
$$
\begin{pmatrix}\tpsi_0\\\tM_0\\\tP_0\end{pmatrix}
\,=\,\begin{pmatrix}\uk\d_x\psi_0\\V_0\\\uU\,V_0\end{pmatrix}
\,+\,\sum_{j\in\{0,1,2\}}
\d_x\left(\,\overline{\iD\d_\xi\tphi_j(0,\cdot)}\,V_0
\,-\,\langle\,\d_\xi^2\tphi_j(0,\,\cdot\,);\uU_x\rangle\,\d_x\psi_0\right)\beta_j(0)\,.
$$

At last, to prove formula~\eqref{hf-data} for the case where $\psi_0$ is not low-frequency, from the reduction of Subsection~\ref{s:slow-reduction} it follows that we only need to show that at our level of description the initial data for modulations systems 
$$
\begin{pmatrix}\ds-\uk(\d_x\psi_0)^{HF}\\\ds
\uU_x(\psi_0)^{HF}\\\ds
\uU\,\uU_x(\psi_0)^{HF}\end{pmatrix}
\qquad\textrm{may be replaced with}\qquad
\begin{pmatrix}0\\\ds
-\left(\uU-\int_0^1\uU\right)\,\d_x\psi_0\\\ds
-\left(\tfrac12\uU^2-\int_0^1\tfrac12\uU^2\right)\,\d_x\psi_0\end{pmatrix}\,.
$$
This is obtained by a repeated use of Lemma~\ref{l:refined-vdC} that implies that initial data that are derivatives of localized data --- in particular high-frequency data --- are negligible. To begin with, an integration by part shows that the former data may be replaced with 
$$
\begin{pmatrix}\ds0\\\ds
-\uU\,(\d_x\psi_0)^{HF}\\\ds
-\tfrac12\uU^2\,(\d_x\psi_0)^{HF}\end{pmatrix}\,.
$$
Then the claim follows from the fact that $(\int_0^1\uU)\,(\d_x\psi_0)^{HF}$, $(\int_0^1\tfrac12\uU^2)\,(\d_x\psi_0)^{HF}$, $(\uU-\int_0^1\uU)\,(\d_x\psi_0)^{LF}$ and $(\tfrac12\uU^2-\int_0^1\tfrac12\uU^2)\,(\d_x\psi_0)^{LF}$ are all high-frequency.

%%%%%%%%%%%%%%%%%%%%%%%%%%%%%%%%%%%%%%%%%%%%%%%%%%%%%%         
%              APPENDIX                              %
%%%%%%%%%%%%%%%%%%%%%%%%%%%%%%%%%%%%%%%%%%%%%%%%%%%%%%

\appendix

\section{Numerical investigation of dispersive spectral stability}\label{s:A_numerics}\label{s:A}

As already discussed in the introduction our study of asymptotic decay and leading-order behavior shall rely on good dispersive properties of the dynamics originating in 
\be
\label{A}
\begin{array}{l}
\textrm{At no nonzero point of spectral curves the second-order}\\
\textrm{derivative with respect to Floquet exponents vanish}\\ 
\textrm{and the third-order derivatives do not vanish at zero.}
\end{array}\tag{A}
\ee

The asymptotic decay part, as stated in Theorem~\ref{th:asymptotic}, may be obtained under the slightly weaker condition~\ref{A_intro}. Yet as variations needed along the proof are mostly of notational order\footnote{As follows from large-eigenvalue asymptotics there are only a finite number of points where the second-order derivatives could possibly vanish so that most of the extra trouble consists in introducing notation for those points and an adapted finite partition of unity in the extended Floquet variable $\xi$.} we perform the full analysis under condition~\ref{A}. This is also consistent with the fact that numerical experiments suggest that condition~\ref{A} does hold for all waves. We give now some pieces of evidence to support this claim. 

Recall first that by symmetries of the equation it is sufficient to analyze a one-parameter family given by the square of cnoidal functions. To ease comparisons corresponding to different periods we plot here the spectrum of operators that have not been scaled to be of period one and hence analyze the spectrum of $\tL_{m,\xi}$ acting through
$$
\tL_{m,\xi}g\,=\,4\times(2m-1)\d_xg+\d_x(12\,m\,\cn^2(\,\cdot\,,m)\,g)-\d_x^3 g
$$
for $\xi$ varying in a suitable Brillouin zone depending on $m$, where $m$ is the square of a elliptic modulus. Also we use "glued" representations of the spectrum as introduced in the kernel representation subsection, Subsection~\ref{s:Green}. At last, note that by Hamiltonian and real symmetries it is sufficient to investigate only the upper half of the spectrum.

In Figure~\ref{line} we show half of the "line" spectrum for $m=0.025$, $0.05$, $0.1$, $0.2$, $0.3$, $0.4$, $0.5$, $0.6$, $0.7$, $0.8$, $0.9$ and $0.95$. In Figure~\ref{third} we plot corresponding third-order derivatives. Likewise for the same parameters we show half of the "loop" spectrum in Figure~\ref{loop} and corresponding second-order derivatives in Figure~\ref{second}. The intermediate Lax spectral parametrization --- used to compute our graphs --- being singular near zero and infinity we avoid to get to close to singularities $0$ at both ends of the half-loop and $0$ and $\iD\infty$ at the end of the half-line.

Beyond the mere observation that condition~\ref{A} appears clearly to hold, a few comments are in order. We stress that one may actually prove most of the following claims by inspecting distinguished limits $m\to0$ (small amplitude), $m\to1$ (homoclinic/solitary wave limit), %$(\lambda,\xi)\to(0,0)$ (slow and side-band) 
or $|\lambda|\to\infty$ (fast).

\begin{enumerate}
\item In the limit $m\to0$, the full spectrum (line and loop) converges locally uniformly far from singularities $|\lambda|=0$ and $|\lambda|=\infty$ to the Fourier spectral curve $\xi\mapsto -4\iD\xi-(\iD\xi)^3=\iD(-4\xi+\xi^3)$. In particular the maximal height of the curve converges to $16/(3\sqrt(3))$ which is approximately $3.0792$, the second order derivative converges to $\xi\mapsto 6\xi$ and the third to $\xi\mapsto 6$.
\item In the limit $m\to1$, the "line" spectrum converges locally uniformly far from singularities to the Fourier spectral curve $\xi\mapsto 4\iD\xi-(\iD\xi)^3=\iD(4\xi+\xi^3)$. In turn, the loop spectrum shrinks to the embedded eigenvalue $0$, associated with invariance by translation (in space).
\item In the large spectrum limit the "line" spectrum is equivalent to $\xi^3$ and the third-order derivative converges to $6$.
\end{enumerate}

\begin{figure}[htbp]
\begin{center}
$
\begin{array}{lr}
(a)\includegraphics[scale=0.25]{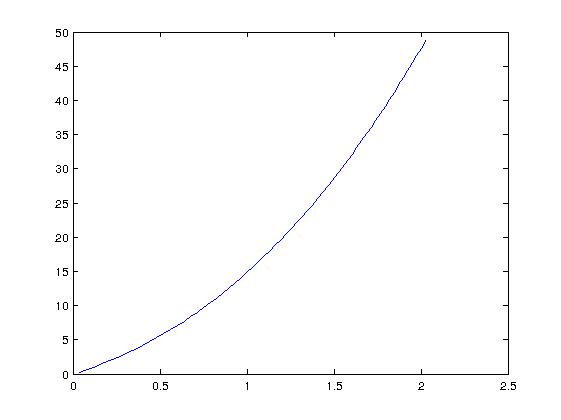}
&
(b)\includegraphics[scale=0.25]{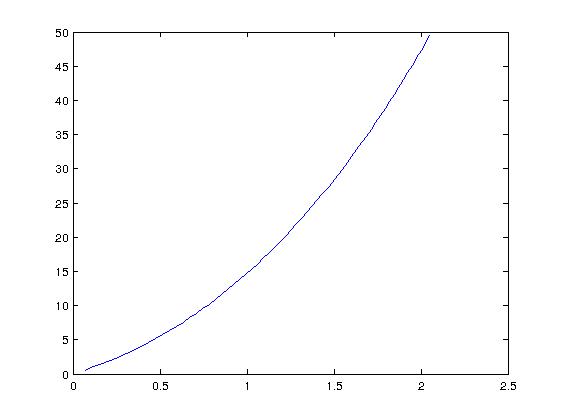}
\\
(c)\includegraphics[scale=0.25]{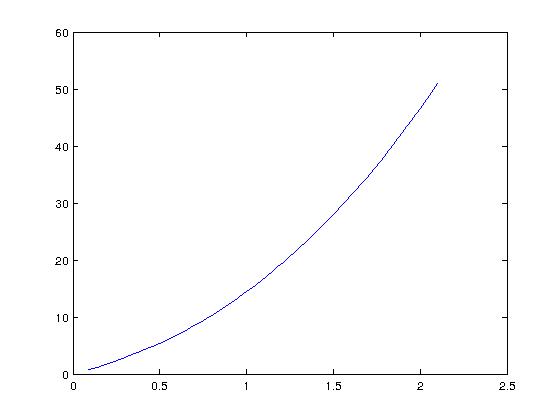}
&
(d)\includegraphics[scale=0.25]{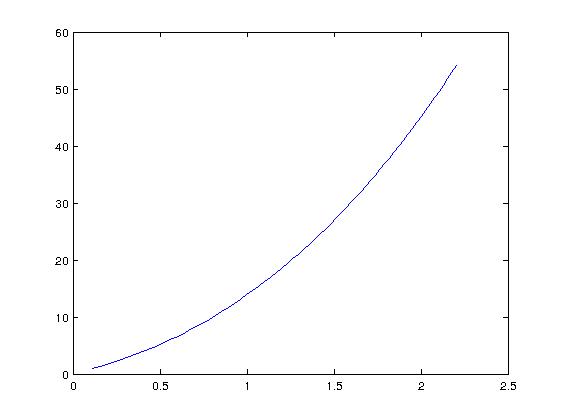}
\\
(e)\includegraphics[scale=0.25]{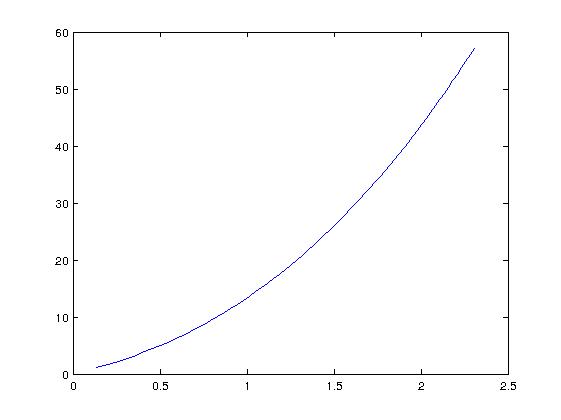}
&
(f)\includegraphics[scale=0.25]{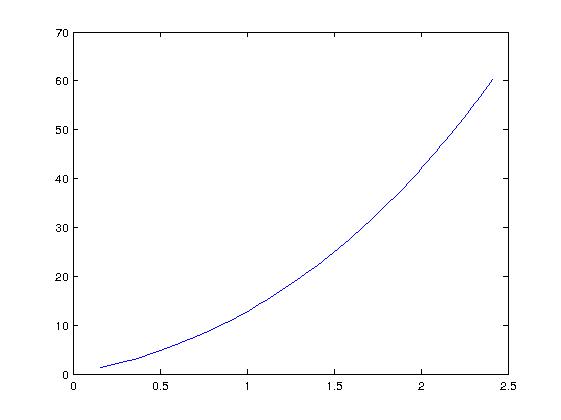}
\\
(g)\includegraphics[scale=0.25]{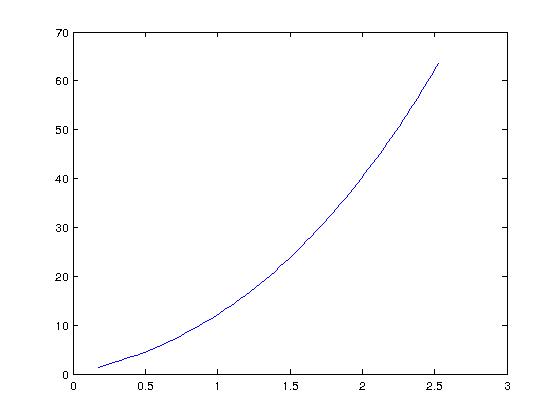}
&
(h)\includegraphics[scale=0.25]{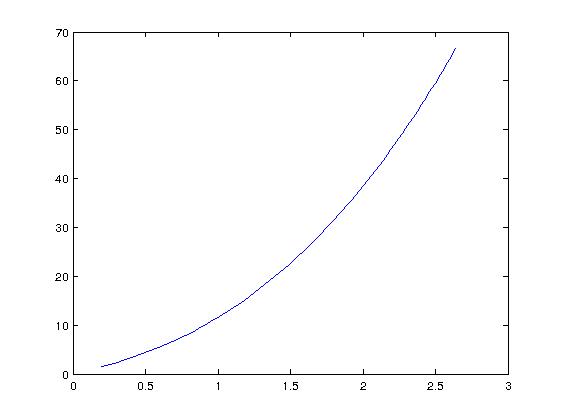}
\\
(i)\includegraphics[scale=0.25]{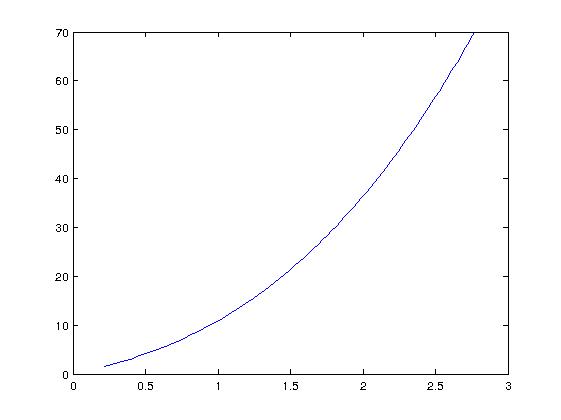}
&
(j)\includegraphics[scale=0.25]{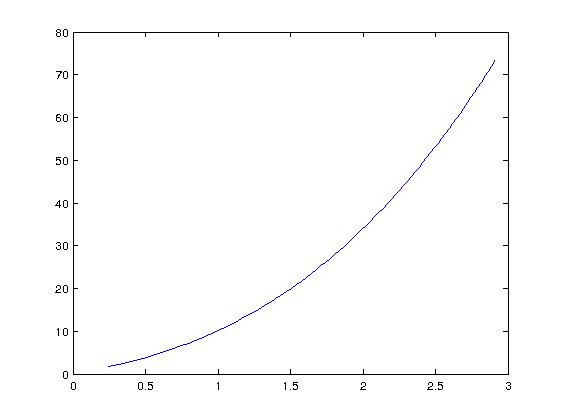}
\\
(k)\includegraphics[scale=0.25]{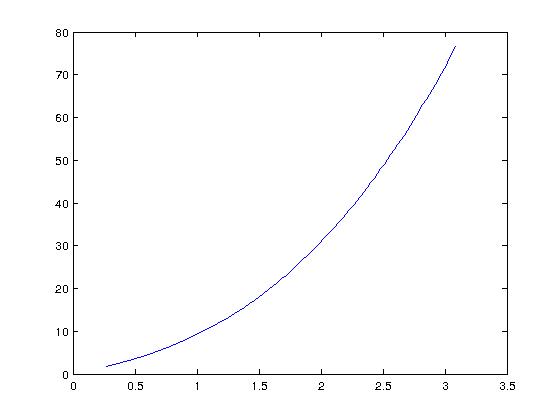}
&
(l)\includegraphics[scale=0.25]{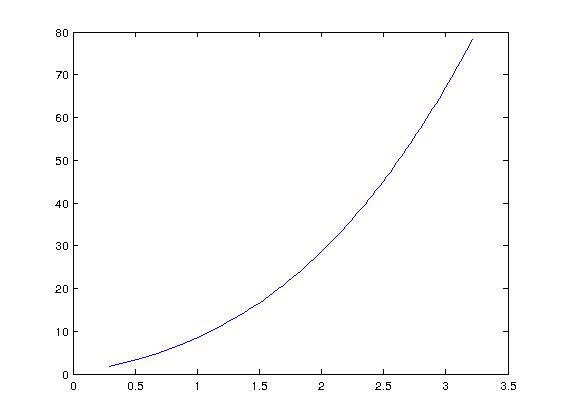}
\end{array}
$
\end{center}
\caption{Imaginary part of the spectrum along the upper half-line versus extended Floquet exponent for $m=0.025$, $0.05$, $0.1$, $0.2$, $0.3$, $0.4$, $0.5$, $0.6$, $0.7$, $0.8$, $0.9$ and $0.95$.}
\label{line}
\end{figure}

\begin{figure}[htbp]
\begin{center}
$
\begin{array}{lr}
(a)\includegraphics[scale=0.25]{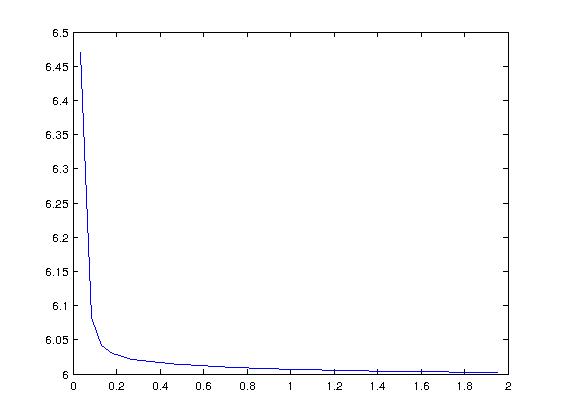}
&
(b)\includegraphics[scale=0.25]{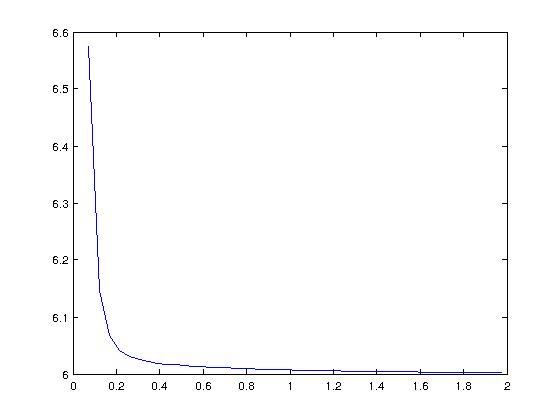}
\\
(c)\includegraphics[scale=0.25]{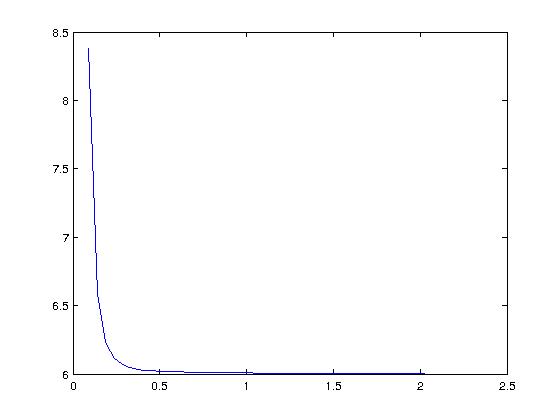}
&
(d)\includegraphics[scale=0.25]{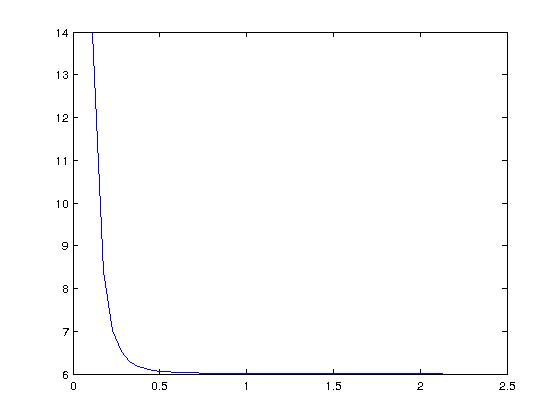}
\\
(e)\includegraphics[scale=0.25]{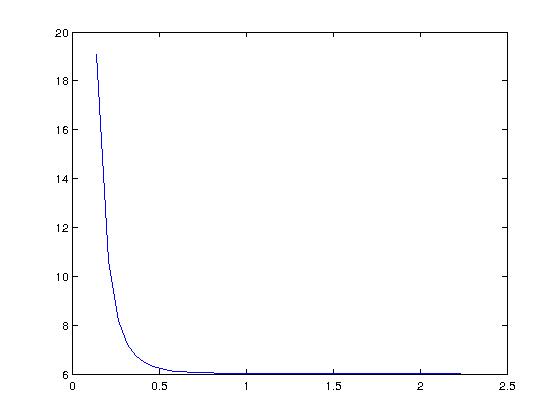}
&
(f)\includegraphics[scale=0.25]{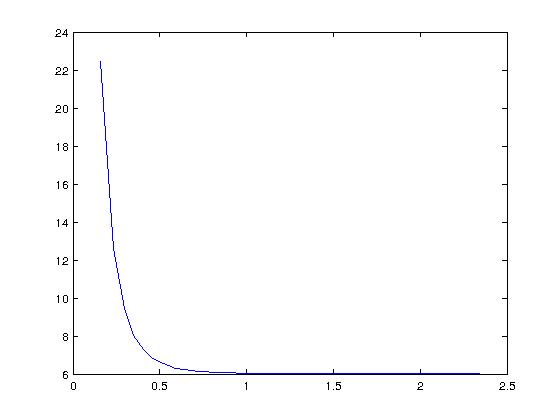}
\\
(g)\includegraphics[scale=0.25]{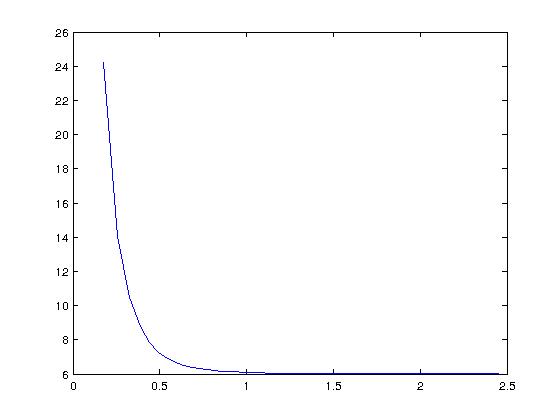}
&
(h)\includegraphics[scale=0.25]{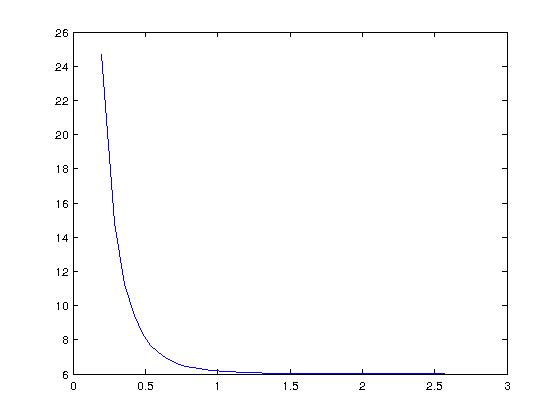}
\\
(i)\includegraphics[scale=0.25]{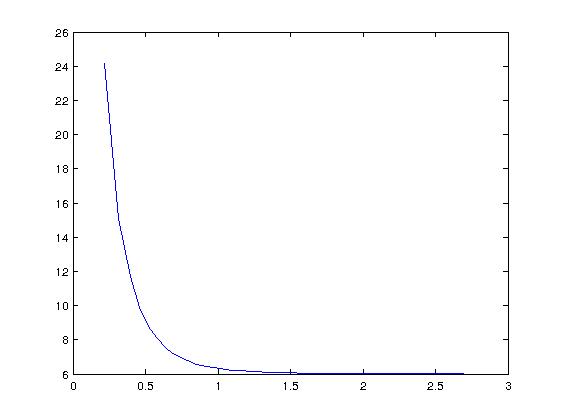}
&
(j)\includegraphics[scale=0.25]{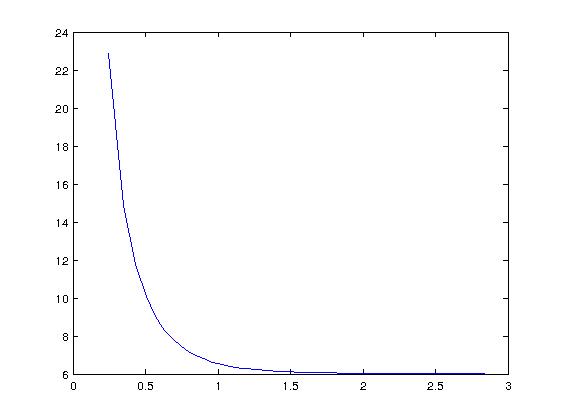}
\\
(k)\includegraphics[scale=0.25]{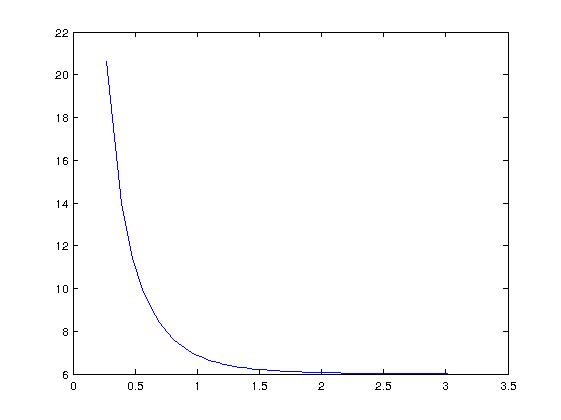}
&
(l)\includegraphics[scale=0.25]{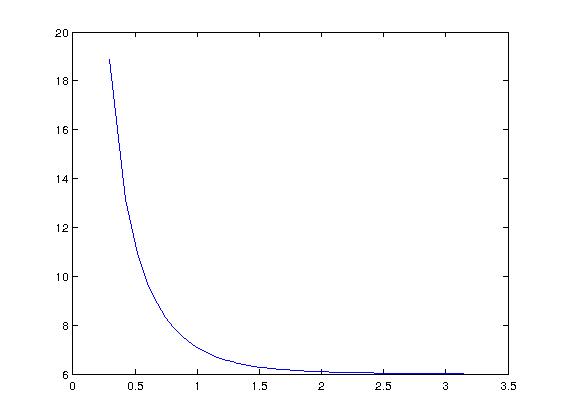}
\end{array}
$
\end{center}
\caption{Third-order derivative with respect to the Floquet exponent of the imaginary part of the spectrum along the upper half-line versus extended Floquet exponent for $m=0.025$, $0.05$, $0.1$, $0.2$, $0.3$, $0.4$, $0.5$, $0.6$, $0.7$, $0.8$, $0.9$ and $0.95$.}
\label{third}
\end{figure}

\begin{figure}[htbp]
\begin{center}
$
\begin{array}{lr}
(a)\includegraphics[scale=0.25]{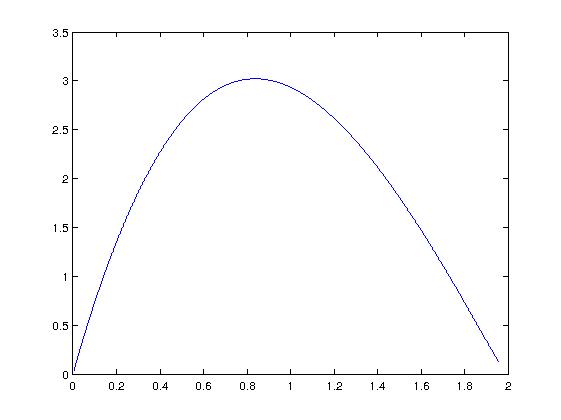}
&
(b)\includegraphics[scale=0.25]{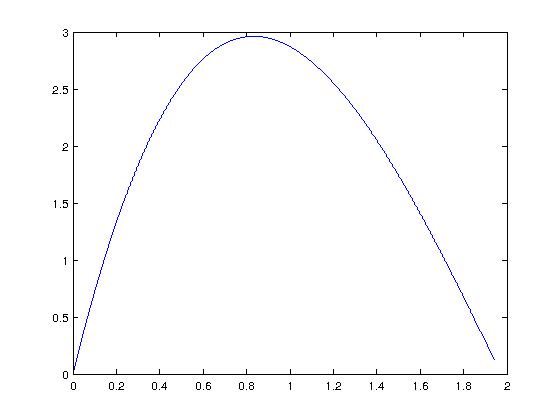}
\\
(c)\includegraphics[scale=0.25]{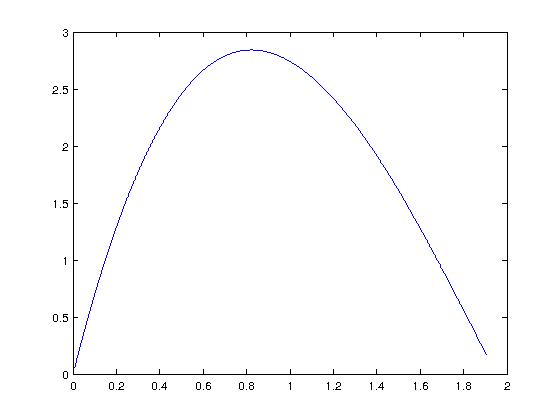}
&
(d)\includegraphics[scale=0.25]{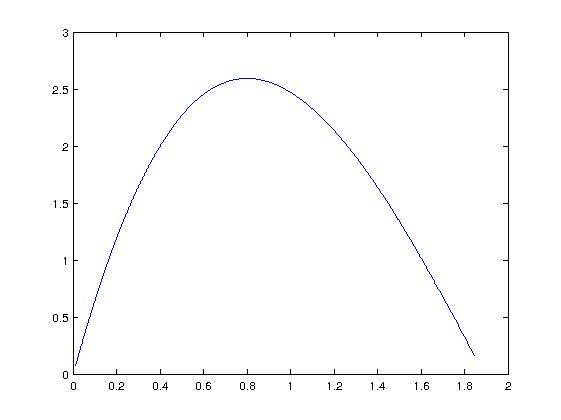}
\\
(e)\includegraphics[scale=0.25]{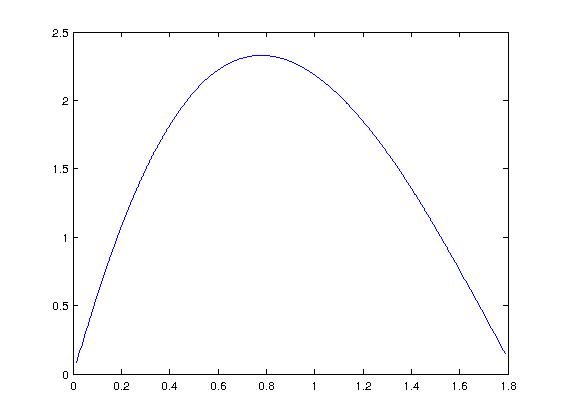}
&
(f)\includegraphics[scale=0.25]{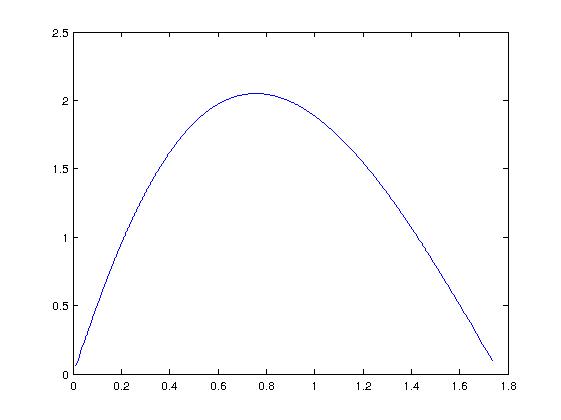}
\\
(g)\includegraphics[scale=0.25]{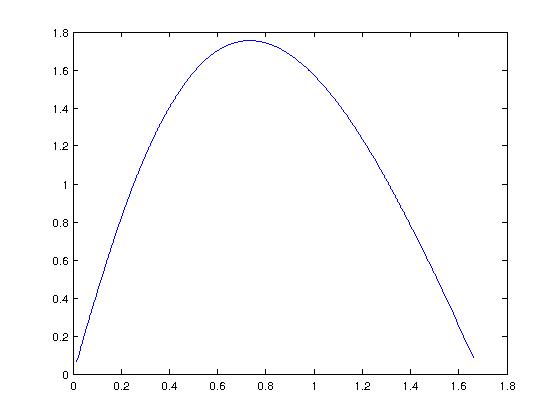}
&
(h)\includegraphics[scale=0.25]{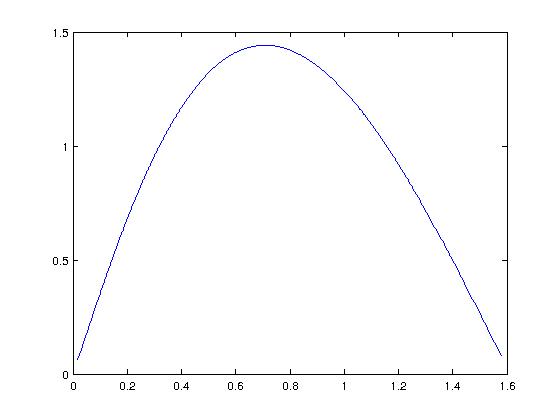}
\\
(i)\includegraphics[scale=0.25]{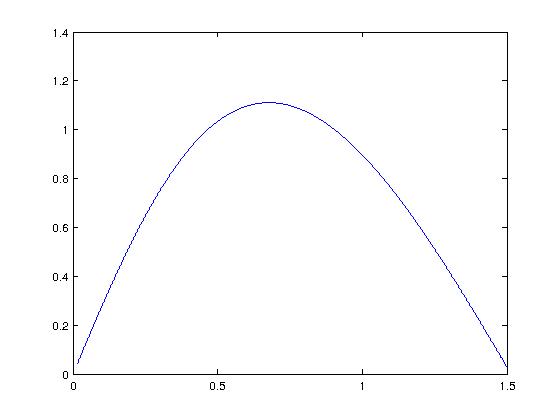}
&
(j)\includegraphics[scale=0.25]{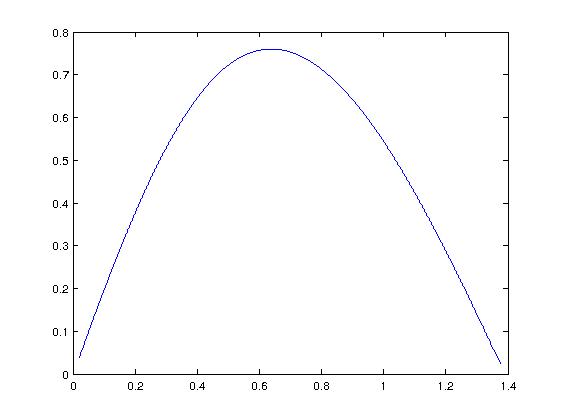}
\\
(k)\includegraphics[scale=0.25]{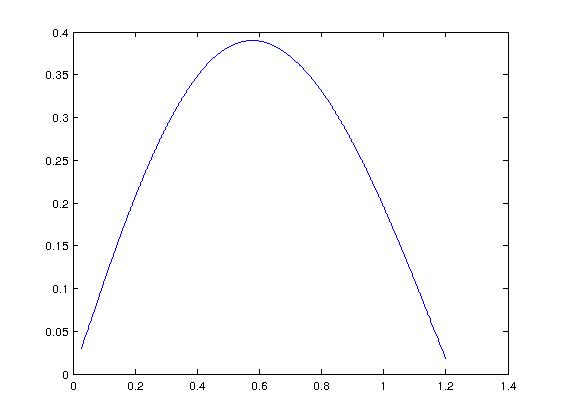}
&
(l)\includegraphics[scale=0.25]{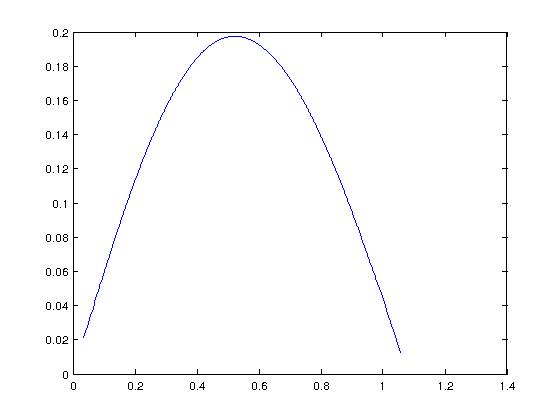}
\end{array}
$
\end{center}
\caption{Imaginary part of the spectrum along the upper half-loop versus extended Floquet exponent for $m=0.025$, $0.05$, $0.1$, $0.2$, $0.3$, $0.4$, $0.5$, $0.6$, $0.7$, $0.8$, $0.9$ and $0.95$.}
\label{loop}
\end{figure}

\begin{figure}[htbp]
\begin{center}
$
\begin{array}{lr}
(a)\includegraphics[scale=0.25]{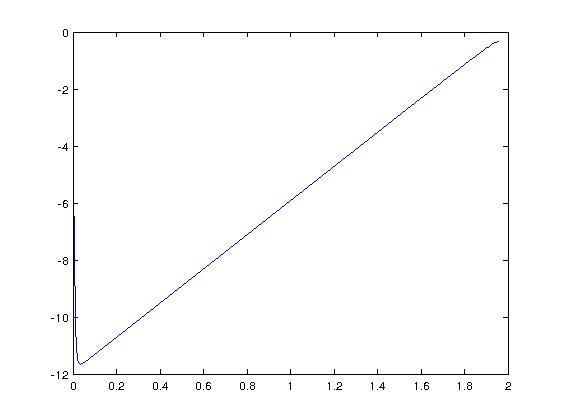}
&
(b)\includegraphics[scale=0.25]{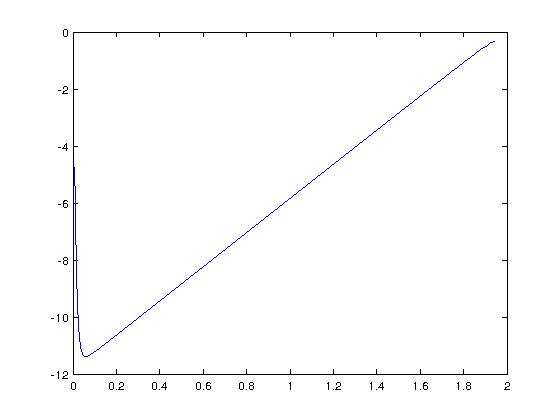}
\\
(c)\includegraphics[scale=0.25]{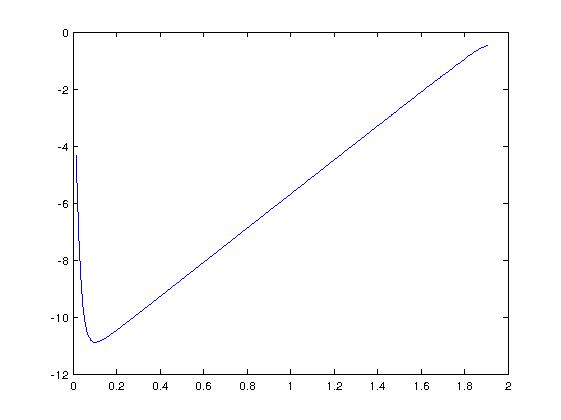}
&
(d)\includegraphics[scale=0.25]{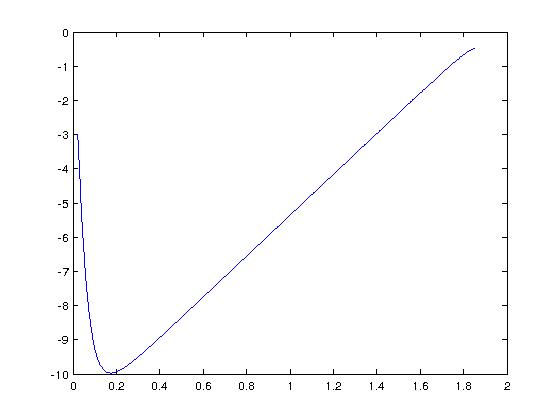}
\\
(e)\includegraphics[scale=0.25]{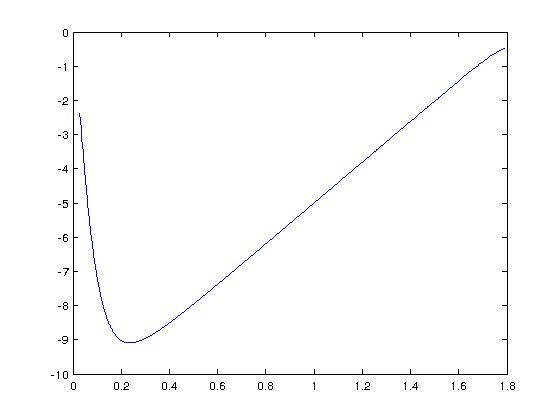}
&
(f)\includegraphics[scale=0.25]{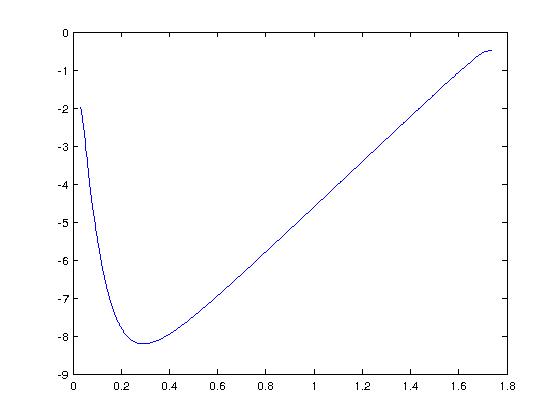}
\\
(g)\includegraphics[scale=0.25]{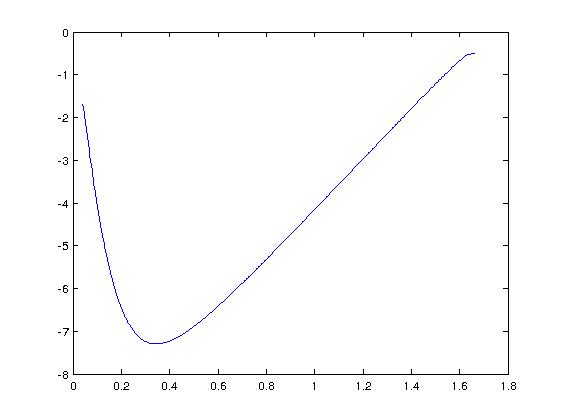}
&
(h)\includegraphics[scale=0.25]{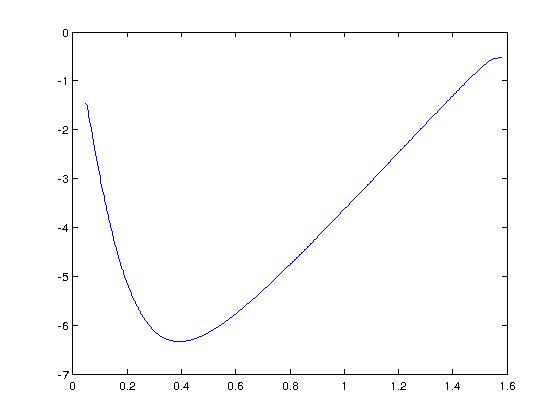}
\\
(i)\includegraphics[scale=0.25]{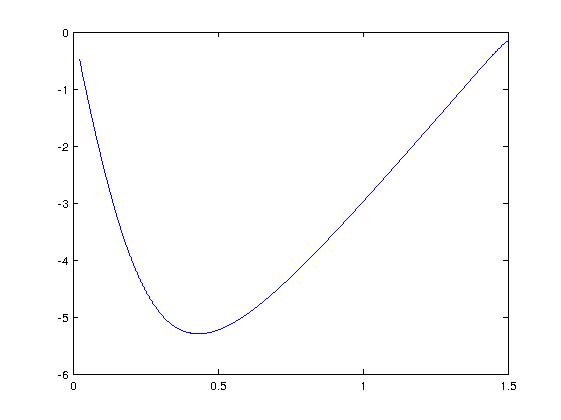}
&
(j)\includegraphics[scale=0.25]{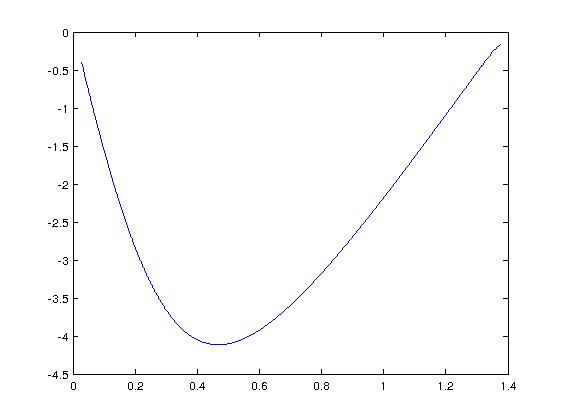}
\\
(k)\includegraphics[scale=0.25]{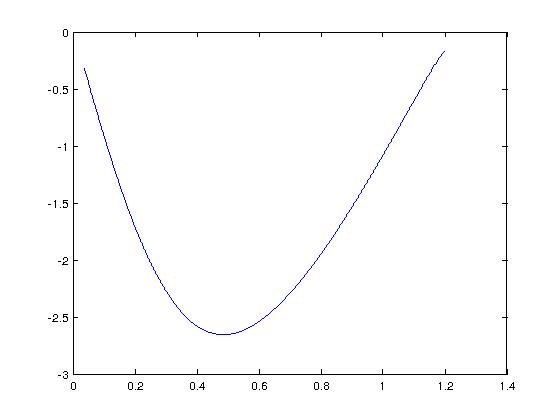}
&
(l)\includegraphics[scale=0.25]{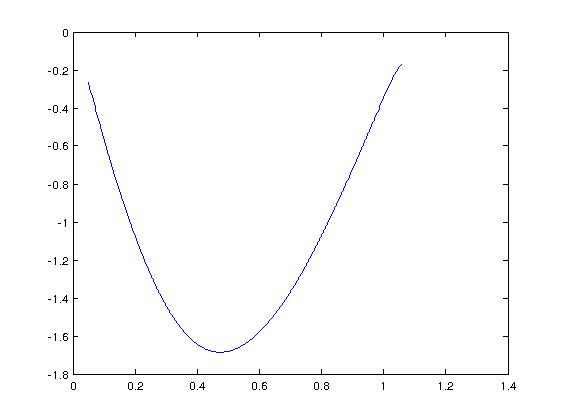}
\end{array}
$
\end{center}
\caption{Second-order derivative with respect to the Floquet exponent of the imaginary part of the spectrum along the upper half-loop versus extended Floquet exponent for $m=0.025$, $0.05$, $0.1$, $0.2$, $0.3$, $0.4$, $0.5$, $0.6$, $0.7$, $0.8$, $0.9$ and $0.95$.}
\label{second}
\end{figure}

%\section{Higher-order averaging}\label{s:order3}

\section{Proofs of oscillatory integrals estimates}\label{s:oscillatory}

For the sake of completeness we provide here proofs of elementary estimates on oscillatory estimates used along our proofs of main results. As a warming up we first give a proof of the classical Lemma~\ref{l:vdC}.

\begin{proof}
We begin with the easiest case when $p=1$. Assume first that $I$ is bounded. An integration by parts gives
$$
\int_I \eD^{\ii\,a}\ F
\,=\,\left[\eD^{\ii\,a}\,\frac{F}{\ii a'}\right]_I
\,-\,\int_I \eD^{\ii\,a}\ \left(\frac{F'}{\ii a'}-\frac{F\,a''}{\ii (a')^2}\right)
$$
which is bounded by
$$
\frac{2\,\sup_I|F|}{\inf_I|a'|}\,+\,\frac{\sup_I|F|\,+\,\int_I|F'|}{\inf_I|a'|}
$$
since monotonicity yields
$$
\int_I\frac{|a''|}{(a')^2}\,\leq\,\frac{1}{\inf_I|a'|}\,.
$$
The case where $I$ is unbounded is then obtained by a simple limiting argument using that $|a'|$ is coercive.

We turn now to the case when $p\geq2$. The assumption implies that $|a'|$ is coercive and monotone at infinity, thus the reduction to the bounded case may be carried out as in the case when $p=1$. Therefore we only deal with the case when $I$ is bounded. For this case we argue by induction. Assume the case $p-1$ has been obtained. The assumption implies that there exists a unique $\xi_*\in I$ minimizing $|a^{(p-1)}|$ on $I$. Now, for any $\delta>0$, the assumptions of the $p-1$ case are satisfied on each connected component of $I_\delta:=I\setminus(\xi_*-\delta,\xi_*+\delta)$ with 
$$
\inf_{I_\delta}\,|a^{(p-1)}|\,\geq\,\delta\,\inf_{I}\,|a^{(p)}|
$$ 
therefore by splitting $I$ as $I_\delta$ and $I\setminus I_\delta$
$$
\left|\int_I \eD^{\ii\,a}\ F\right|
\,\leq\,\frac{2\,C_{p-1}}{(\delta\,\inf_{I}\,|a^{(p)}|)^{1/(p-1)}}\ \left[\,\sup_I|F|\,+\,\int_I|F'|\,\right]\,+\,2\,\delta\,\sup_I|F|
$$
where $C_{p-1}$ denotes the constant obtained in the $(p-1)$ case. Setting\footnote{This is motivated by a minimization in $\delta$ of the estimate but provides the optimal value only up to a multiplying constant that depends only on $p$. As is customary we chose a non optimal value to receive expressions as simple as possible. We shall repeat the same pattern often in the present section.} 
$$
\delta\,=\,\left(\inf_{I}\,|a^{(p)}|\right)^{-1/p}
$$
in the above estimate achieves the proof.
\end{proof}

We now prove Lemma~\ref{l:refined-vdC}.

\begin{proof}
For writing convenience we set
$$
K\ =\ \sup_I|F|\,+\,\int_I|F'|\,+\,\sup_I|G|\qquad
\textrm{and}\qquad
A\ =\ \inf_I|a^{(p)}|\,.
$$
If $A\leq1$ the results follows directly from Lemma~\ref{l:vdC} so we assume $A\geq1$. We detail only the case where $\xi_*$ belongs to the boundary of $I$. The general result is then obtained by applying this special case to $I\cap[\xi_*,\infty[$ and $I\cap ]-\infty,\xi_*]$ and doubling the constant derived in the special case.

In this special case it follows that $|a^{(p-2)}|$ vanishes at most once and that $a^{(p-2)}$ is monotone. If it exists we denote the point where $|a^{(p-2)}|$ vanishes as $\xi_0$. If either $|a^{(p-2)}|$ does not vanish or if $|\xi_0-\xi_*|\leq \delta_1$ with
$\delta_1=A^{-1/(p+\alpha\,(p-2))}$ we derive
$$
\inf_{I\cap (\R\setminus(\xi_*-2\,\delta_1,\xi_*+2\,\delta_1))} |a^{(p-2)}| \,\geq\, \tfrac32\,A\,\delta_1^2
$$
that yields the bound
$$
\frac{K}{\alpha+1}\,(2\delta_1)^{\alpha+1}\,+\,\frac{C\,K}{(\tfrac32\,A\,\delta_1^2)^{1/(p-2)}}\,=\,\left(\frac{2^{\alpha+1}}{\alpha+1}\,+\,C\,\left(\frac23\right)^{\tfrac{1}{p-2}}\right)\,K\,A^{-(\alpha+1)/(p+\alpha\,(p-2))}
$$
where $C$ is obtained from the $p-2$ case of Lemma~\ref{l:vdC}. When there is a zero $\xi_0$ and it satisfies $\delta_1\leq |\xi_0-\xi_*|\leq 1$ we use that when $0<\delta_2\leq|\xi_0-\xi_*|$ one may derive 
$$
\inf_{I\cap (\R\setminus(\xi_0-\delta_2,\xi_0+\delta_2))} |a^{(p-2)}| \,\geq\, \tfrac12\,A\,|\xi_0-\xi_*|\,\delta_2
$$
that yields the bound
$$
2\,K\,(2\,|\xi_0-\xi_*|)^{\alpha}\,\delta_2\,+\,\frac{2\,C\,K}{(\tfrac12\,A\,|\xi_0-\xi_*|\,\delta_2)^{1/(p-2)}}
$$
which by choosing 
$$
\delta_2\,=\,A^{-1/(p-1)}\, |\xi_0-\xi_*|^{-(1+\alpha\,(p-2))/(p-1)}
$$
leads to 
$$
\begin{array}{rcl}\ds
\left(2^{\alpha+1}\,+\,2\,C\right)K\,A^{-1/(p-1)}\,|\xi_0-\xi_*|^{(\alpha-1)/(p-1)}
&\leq&\ds
\left(2^{\alpha+1}\,+\,2\,C\right)K\,A^{-1/(p-1)}\,\delta_1^{(\alpha-1)/(p-1)}\\
&&\,=\ds\left(2^{\alpha+1}\,+\,2\,C\right)K\,A^{-(\alpha+1)/(p+\alpha\,(p-2))}\,.
\end{array}
$$
Note that the foregoing choice of $\delta_2$ is indeed available since $\delta_2\leq |\xi_0-\xi_*|$ follows from
$$
|\xi_0-\xi_*|\geq \delta_1\geq A^{-1/(p+\alpha\,(p-2))}\,.
$$
At last, when there is a zero $\xi_0$ and it satisfies $|\xi_0-\xi_*|\geq 1$ we alternatively rely on the fact that for any $0<\delta'_2\leq|\xi_0-\xi_*|$ we may also bound the quantity of interest by
$$
2\,K\,\delta'_2\,+\,\frac{2\,C\,K}{(\tfrac12\,A\,|\xi_0-\xi_*|\,\delta'_2)^{1/(p-2)}}
$$
then choose 
$$
\delta'_2\,=\,A^{-1/(p-1)}\, |\xi_0-\xi_*|^{-1/(p-1)}
$$
to derive the bound 
$$
2\,\left(1\,+\,C\right)K\,A^{-1/(p-1)}\,|\xi_0-\xi_*|^{-1/(p-1)}
\,\leq\,2\,\left(1\,+\,C\right)K\,A^{-1/(p-1)}\,.
$$
Note that the latter choice of $\delta'_2$ is legitimate since $\delta'_2\leq |\xi_0-\xi_*|$ stems from
$$
|\xi_0-\xi_*|\geq \delta_1\geq A^{-1/p}
$$
and that the dependence on $A$ of latter bound is indeed as least as good as the one obtained in the first step since $\alpha\leq1$ yields 
$$
\frac{\alpha+1}{p+\alpha\,(p-2)}\,\leq \frac{1}{p-1}\,.
$$
\end{proof}

An examination of the previous proof shows that the case $\alpha>1$ would yield the same bound as in the case $\alpha=1$ and that this bound is optimal.

We begin the proof of our last oscillating integral lemmas by providing a control on stagnating points in terms of variations of phases.

\bl\label{l:vanishing}
Let $I$ be an open interval, $\xi_*\in I$, $\rho>2$, $M\geq0$ and $\kappa>0$. There exist positive $\eps_0$ and $C_0$ such that if 
\begin{itemize}
\item $A:I\to \R$ and $B:I\to\R$ are $\cC^2$, and such that $A''$ and $B''$ are lower-bounded by $\kappa$ ;
\item $A'(\xi_*)=B'(\xi_*)=0$ ;
\item for any $\xi\in I$, $|A(\xi)-B(\xi)|\leq M|\xi-\xi_*|^\rho$
\end{itemize}
then if there exists $\xi_0\in I$ such that $A(\xi_0)=0$ and $|\xi_0-\xi_*|\leq \eps_0$ there also exists $\xi_1\in I$ such that $B(\xi_1)=0$ and $|\xi_0-\xi_1|\leq C_0 |\xi_0-\xi_*|^{\rho-1}$.
\el

\begin{proof}
We first pick $\eps_0'$ such that $[\xi_*-2\eps_0',\xi_*+2\eps_0']\subset I$.
Let $\xi_0$ be such that $A(\xi_0)=0$ and $|\xi_0-\xi_*|\leq \eps_0'$. Since $\rho>0$, if $\xi_0=\xi_*$ then $B(\xi_*)=0$ so that we may focus on the case where $\xi_0\neq \xi_*$. Then let $\sigma$ denote the sign of $\xi_0-\xi_*$. When $0\leq r\leq|\xi_0-\xi_*|$ on one hand
$$
\,A(\xi_0+\sigma\,r)\geq \tfrac12 \kappa(2|\xi_0-\xi_*|+r)r\geq \kappa\,|\xi_0-\xi_*|\,r
$$
and
$$
A(\xi_0-\sigma\,r)\leq -\tfrac12\kappa\,(2|\xi_0-\xi_*|-r)\,r\leq -\tfrac12\kappa\,|\xi_0-\xi_*|\,r
$$
and on the other hand
$$
|A(\xi_0+\sigma\,r)-B(\xi_0+\sigma\,r)|\leq
M\,(|\xi_0-\xi_*|+r)^\rho
$$
and
$$
|A(\xi_0-\sigma\,r)-B(\xi_0-\sigma\,r)|\leq
M\,(|\xi_0-\xi_*|-r)^\rho\,.
$$
From here the Intermediate Value Theorem yields the result provided one chooses $r=C_0|\xi_0-\xi_*|^{\rho-1}$ and ensures
$$
C_0|\xi_0-\xi_*|^{\rho-1}\leq |\xi_0-\xi_*|\,,\qquad
M\leq \tfrac12\kappa C_0
\qquad\textrm{and}\qquad
M\,2^\rho<\kappa C_0
$$
when $|\xi_0-\xi_*|\leq \eps_0$ for some $0<\eps_0\leq\eps_0'$. The latter may indeed be achieved by first choosing $C_0$ large enough to satisfy the two last constraints then $\eps_0$ small enough, since $\rho>2$.
\end{proof}

We now prove Lemma~\ref{l:phase-vdC}.

\begin{proof}
First we pick $\eps_0$ and $C_0$ given by Lemma~\ref{l:vanishing} applied to $A=\omega^{(p-2)}$ and $B=\omega_0^{(p-2)}$ with $\rho=q-p+2$. If necessary we restrict $\eps_0$ further to ensure that $C_0\eps_0^{\rho-2}\leq 1/4$. As in previous proofs without loss of generality we assume that $\xi_*$ belongs to the boundary of $I$. Also we denote by $\xi_0$ a possible zero of $\omega^{(p-2)}$ and when such a zero exists by $\xi_1$ the corresponding zero of $\omega_0^{(p-2)}$. Note in particular that corresponding $\xi_0$ and $\xi_1$ lie indeed on the same side of $\xi_*$.

For notational convenience, unlike what we have done so far we do not track here dependences on harmless constants and only focus on powers of $t$. We may also focus on the case where $t$ is large since a crude estimate shows that the quantity of interest is bounded by a multiple of $|t|$. We shall do so without mention from now on.

For any $\delta>0$ if $\omega^{(p-2)}$ does not vanish or, if it does, when $2|\xi_0-\xi_*|<\delta$ one obtains a bound by a multiple of $(\delta^2 |t|)^{-1/(p-2)}+\delta^{q+1}|t|$. Choosing then $\delta$ as a multiple of $|t|^{-(p-1)/(p+q(p-2))}$ one derives a bound decaying as $|t|^{-(q-1)/(p+q(p-2))}$ when either there is no $\xi_0$ or $|\xi_0-\xi_*|$ is smaller than some multiple of $|t|^{-(p-1)/(p+q(p-2))}$. This is indeed better than the claimed bound since $q(p-1)-(p+q(p-2))=q-p>0$. 

In turn, for any positive $\delta$ such that $2C_0|\xi_0-\xi_*|^{q-p+1}\leq\delta\leq |\xi_0-\xi_*|/4$ one may bound the studied quantity by a multiple of $(\delta|\xi_0-\xi_*| |t|)^{-1/(p-2)}+\delta |\xi_0-\xi_*|^q|t|$. From this, choosing $\delta$ as a sufficiently small multiple of $|t|^{-1}|\xi_0-\xi_*|^{-(q(p-2)+1)/(p-1)}$ provides a bound by a multiple of $|\xi_0-\xi_*|^{(q-1)/(p-1)}$ provided that $|\xi_0-\xi_*|$ is larger than some arbitrary small multiple of $|t|^{-(p-1)/(p+q(p-2))}$ and smaller than some multiple of $|t|^{-(p-1)/((2q-p+1)(p-2)+q-p+2)}=|t|^{-(p-1)/(p+q(p-2)+(q-p)(p-1))}$. 

Alternatively for any $\delta$ such that $0<\delta\leq |\xi_0-\xi_*|/4$ one may also obtain a bound like $(\delta|\xi_0-\xi_*| |t|)^{-1/(p-2)}+\delta$. From here choosing $\delta$ as a multiple of $(|t||\xi_0-\xi_*|)^{-1/(p-1)}$ yields a bound by a multiple of $(|t||\xi_0-\xi_*|)^{-1/(p-1)}$ provided that $|\xi_0-\xi_*|$ is larger than some multiple of $|t|^{-1/p}$. 

One concludes the proof by using the first estimate when there is no vanishing or $|\xi_0-\xi_*|$ is smaller than some multiple of $|t|^{-(p-1)/(p+q(p-2))}$, the second estimate from there to a multiple of $|t|^{-1/q}$ and the last one in the remaining zone. That this is indeed possible follows from the observation that $(p-1)/(p+q(p-2))=(p-1)/(q(p-1)-(q-p))>1/q$ and $(p-1)/(p+q(p-2)+(q-p)(p-1))=(p-1)/(q(p-1)+(q-p)(p-2))\leq 1/q$.
\end{proof}

We now prove Lemma~\ref{l:singular-phase-vdC}.

\begin{proof}
We chose $\eps_0$ and $C$ as in the foregoing proof and adopt the same simplifying convention. We note however that here instead of bounding integrals far from stagnating points by appealing to the van der Corput Lemma we use sharper bounds following from an inspection of its proof. To be more concrete note that one may replace bounds like $(\delta'\delta|\xi_0-\xi_*|\,|t|)^{-1}$ for integrals where $|\xi-\xi_*|>\delta'$ and $|\xi-\xi_0|>\delta$ obtained when $4\delta <|\xi_0-\xi_*|$ and $4\delta'<|\xi_0-\xi_*|$ by a bound by a multiple of a sum of
$$
\int_{\delta'}^{|\xi_0-\xi_*|-\delta} \frac{\dD\zeta}{\zeta^2(|\xi_0-\xi_*|^2-\zeta^2)|t|}
$$
and similar integral and boundary terms that are bounded by a multiple of 
$$
(\delta'|\xi_0-\xi_*|^2|t|)^{-1}+(\delta|\xi_0-\xi_*|^2|t|)^{-1}\,.
$$

As in the proof of Lemma~\ref{l:phase-vdC}, for any $\delta>0$ if $\omega'$ does not vanish or otherwise if its zero $\xi_0$ is such that $2|\xi_0-\xi_*|<\delta$ one obtains a bound by a multiple of $\delta^{-1}(\delta^2 |t|)^{-1}+\delta^q|t|$. Choosing then $\delta$ as a multiple of $|t|^{-2/(q+3)}$ one derives a bound decaying as $|t|^{-(q-3)/(q+3)}$ when either there is no $\xi_0$ or $|\xi_0-\xi_*|$ is smaller than some multiple of $|t|^{-2/(q+3)}$. This is indeed better than the claimed bound $(q-3)/(q+3)>(q-3)/2q$. 

With our preliminary remark in mind, we now observe that for any positive $\delta$ and $\delta'$ such that $2C_0|\xi_0-\xi_*|^{q-2}\leq\delta\leq |\xi_0-\xi_*|/4$ and $\delta'\leq |\xi_0-\xi_*|/4$ one may bound the quantity of interest by a multiple of 
$$
(\delta')^q|t|+(\delta'|\xi_0-\xi_*|^2|t|)^{-1}+(\delta|\xi_0-\xi_*|^2|t|)^{-1}+\delta |\xi_0-\xi_*|^{q-1}|t|\,.
$$
From this choosing $\delta$ and $\delta'$ as sufficiently small multiples of respectively $|t|^{-1}|\xi_0-\xi_*|^{-(q+1)/2}$ and $(|\xi_0-\xi_*||t|)^{-2/(q+1)}$ provides a bound by a multiple of 
$$
|\xi_0-\xi_*|^{(q-3)/2}+(|\xi_0-\xi_*|^{2q}|t|^{q-1})^{-1/(q+1)}
$$
provided that $|\xi_0-\xi_*|$ is larger than some arbitrary small multiple of $|t|^{-2/(q+3)}$ and smaller than some multiple of $|t|^{-2/(3(q-1))}$. Note that, by convexity, at any fixed time $t$, when $|\xi_0-\xi_*|$ varies in an interval the latter bound reaches is upper bound on the boundary of the interval.

At last when $\delta$ and $\delta'$ are such that $0<\delta\leq |\xi_0-\xi_*|/4$ and $0<\delta'\leq |\xi_0-\xi_*|/4$ one may also obtain a bound like 
$$
(\delta')^q|t|+(\delta'|\xi_0-\xi_*|^2|t|)^{-1}+(\delta|\xi_0-\xi_*|^2|t|)^{-1}+\delta |\xi_0-\xi_*|^{-1}\,.
$$
From here choosing $\delta$ as a multiple of $(|t||\xi_0-\xi_*|)^{-1/2}$ and $\delta'$ as sufficiently small multiple of $(|\xi_0-\xi_*||t|)^{-2/(q+1)}$ yields a bound by a multiple of $$
(|t||\xi_0-\xi_*|^3)^{-1/2}+(|\xi_0-\xi_*|^{2q}|t|^{q-1})^{-1/(q+1)}
$$ 
provided that $|\xi_0-\xi_*|$ is larger than some multiple of $|t|^{-2/(q+3)}$. Actually in the previous regime the latter bound is always smaller than a multiple of $(|t||\xi_0-\xi_*|^3)^{-1/2}$.

We conclude the proof by using the first estimate when there is no vanishing or $|\xi_0-\xi_*|$ is smaller than some multiple of $|t|^{-2/(q+3)}$, the second estimate from there to a multiple of $|t|^{-1/q}$ and the last one in the remaining zone. That this is indeed possible follows from the observation that $2/(q+3)=2/(2q-(q-3)))>1/q$ and $2/(3(q-1))=2/(2q+q-3)<1/q$.
\end{proof}

%%%%%%%%%%%%%%%%%%%%%%%%%%%%%%%%%%%%%%%%%%%%%%%%%%%%%% 
%        acknowledgment                              %
%%%%%%%%%%%%%%%%%%%%%%%%%%%%%%%%%%%%%%%%%%%%%%%%%%%%%%

\medskip

\noindent{\it Acknowledgment.} The author would like to warmly thank Corentin Audiard for enlightening exchanges at an early stage of the present project.

%%%%%%%%%%%%%%%%%%%%%%%%%%%%%%%%%%%%%%%%%%%%%%%%%%%%%% 
%        bibliography                                %
%%%%%%%%%%%%%%%%%%%%%%%%%%%%%%%%%%%%%%%%%%%%%%%%%%%%%%

\bibliographystyle{abbrv}
\bibliography{Ref} 

\end{document}